\documentclass[11pt, reqno]{amsart}

\usepackage[T1]{fontenc}
\usepackage[utf8]{inputenc}
\usepackage{lmodern}
\usepackage[a4paper,margin=25mm]{geometry}
\usepackage{amsmath,amssymb,amsthm}
\usepackage{mathtools}
\usepackage{booktabs}
\usepackage[expansion=false]{microtype}
\usepackage[hidelinks]{hyperref}
\usepackage{cite}
\usepackage{graphics}
\usepackage{bm}

\newtheorem{theorem}{Theorem}
\newtheorem{proposition}{Proposition}

\newtheorem{lemma}{Lemma}
\newtheorem{remark}{Remark}
\newtheorem{definition}{Definition}
\newcommand{\tr}{\operatorname{tr}}
\newcommand{\diag}{\operatorname{diag}}
\newcommand{\one}{\underline{1}}

\allowdisplaybreaks

\title[Irregular Cospectral Graphs with Distinct Lov\'{a}sz Numbers]
{Irregular Cospectral Graphs with Distinct Lov\'{a}sz Numbers}

\author[I. Sason]{Igal Sason}

\address{\normalfont \newline 
Igal Sason is with the Viterbi Faculty of Electrical and Computer Engineering and the Faculty
of Mathematics, Technion--Israel Institute of Technology, Haifa 3200003, Israel.
E-mail: \texttt{eeigal@technion.ac.il.}}

\begin{document}

\begin{abstract}
For every integer 
\texorpdfstring{\(n\geq 11\)}{n >= 11}, we construct a pair of connected,
irregular, nonisomorphic graphs on \texorpdfstring{\(n\)}{n} vertices 
that are cospectral with respect to the adjacency, Laplacian, signless Laplacian,
normalized Laplacian, and Seidel matrices and have equal independence, clique,
chromatic, complement chromatic, and maximum-cut numbers, but distinct Lov\'{a}sz 
\texorpdfstring{\(\vartheta\)}{theta}-numbers. The graphs are obtained by joining
each member of a fixed pair of cospectral, nonisomorphic regular graphs
on ten vertices, due to van Dam and Haemers~(2003),
with the complete graph \texorpdfstring{\(K_{n-10}\)}{Kₙ₋₁₀}. We prove that 
this construction preserves the equality of the five spectra and the stated
integer-valued combinatorial invariants within each pair, while leaving
the respective Lov\'{a}sz numbers unchanged. We derive exact formulas
for these Lov\'{a}sz numbers and prove that they are distinct. We also
determine the cardinality-constrained maximum-cut profiles of the two
base graphs and their complements. These profiles yield exact formulas
and establish equality of the maximum-cut numbers within each pair,
both for the joins of the base graphs with \texorpdfstring{\(K_{n-10}\)}{Kₙ₋₁₀} 
and for the joins of their complements with \texorpdfstring{\(K_{n-10}\)}{Kₙ₋₁₀}. 
For \texorpdfstring{\(n=10\)}{n = 10}, we first
present a regular pair satisfying all the stated properties except
irregularity, and then a connected, irregular, nonisomorphic pair that
shares all five spectra and all the stated integer-valued combinatorial
invariants, but has distinct Lov\'{a}sz numbers. An exhaustive computer search 
shows that no pair of connected, irregular, nonisomorphic
graphs on at most nine vertices shares the first four spectra and 
all the stated integer-valued combinatorial invariants. Hence, 
ten is the smallest possible order, and pairs with all the stated properties 
exist for every \texorpdfstring{\(n\geq 10\)}{n >= 10}. 
This extends and strengthens an earlier result for \emph{even}
\texorpdfstring{\(n\geq 14\)}{n >= 14} (Sason, 2024), which used
the neighbors splitting join operation and did not address Seidel matrices, 
complement chromatic numbers, the maximum-cut numbers of the graphs, or 
the maximum-cut numbers of the corresponding joins formed from their complements.
Thus, the Lov\'{a}sz number provides an efficiently computable certificate of 
nonisomorphism even when the five spectra and all the stated integer-valued 
combinatorial invariants coincide.
\end{abstract}

\keywords{
Algebraic graph theory; cospectral graphs; graph invariants; graph isomorphism; 
graph joins; Lov\'{a}sz $\vartheta$-function; spectral graph theory. \vspace*{0.1cm}
}

\maketitle
\thispagestyle{empty}
\setcounter{page}{1}

\section{Introduction}

The Lov\'{a}sz $\vartheta$-function connects graph structure, optimization,
and zero-error information theory. It upper-bounds both the independence
number and the Shannon capacity of a finite, simple, and undirected graph,
lower-bounds the chromatic number of the complement, which equals the vertex 
clique cover number of the graph (the minimum number of cliques needed to 
cover its vertex set), and yields lower bounds on the maximum-cut number of 
the complement \cite{BallaJS2024,Knuth94,Lovasz1979,Lovasz19}. Moreover, its
semidefinite-programming formulation permits approximation to within any
additive error $\varepsilon\in(0,1)$ in time polynomial in the graph order
and $\log(1/\varepsilon)$ \cite{Lovasz19,GrotschelLovaszSchrijver1981}. 
This computational accessibility contrasts with the NP-hardness of exactly 
computing the independence, clique, chromatic, and maximum-cut numbers. 
The broader role of semidefinite programming in graph theory is illustrated 
by Ba\v{c}\'{i}k and Mahajan \cite{BacikMahajan1995}, who relate the
Lov\'{a}sz number to semidefinite relaxations of graph-homomorphism
problems. As the constructions in this paper demonstrate, the
$\vartheta$-number can also reveal structural distinctions that are
invisible to these familiar integer-valued combinatorial invariants and 
to several standard graph spectra.

From the perspective of graph isomorphism, graph invariants provide
necessary conditions for isomorphism: unequal invariant values certify
nonisomorphism, whereas equal values do not, in general, establish
isomorphism. Spectral graph determination asks when the spectrum of a
matrix associated with a graph determines the graph up to isomorphism.
The surveys of van Dam and Haemers \cite{vanDamHaemers2003,vanDamHaemers2009} 
discuss this question for several graph matrices and review numerous 
spectral characterizations. Related methods for constructing irregular, 
cospectral, nonisomorphic graphs appear in the work of Hamud and Berman
\cite{HamudB24} and in the references therein. More recently, Sason, 
Krupnik, Hamud, and Berman \cite{SasonKrupnikHamudBerman2025} surveyed 
classical and recent developments, including methods for constructing 
and distinguishing cospectral nonisomorphic graphs and conditions under 
which graphs are determined by their spectra. These works provide a broader 
context for investigating which structural information is captured by 
graph spectra and which additional invariants can distinguish graphs 
having identical spectra.

The constructions presented here show that connected nonisomorphic
graphs can share the adjacency, Laplacian, signless Laplacian,
normalized Laplacian, and Seidel spectra, as well as the independence,
clique, chromatic, complement chromatic, and maximum-cut numbers,
yet have distinct Lov\'{a}sz numbers. The Lov\'{a}sz number therefore
provides a semidefinite certificate of nonisomorphism when these
spectral and combinatorial invariants fail to distinguish the graphs.
For the infinite family constructed in this paper, exact analytical
calculations establish that the two Lov\'{a}sz numbers are distinct
and individually independent of the graph order. The five common
spectra and the common clique, chromatic, and maximum-cut numbers
vary with the graph order, whereas the common independence number
and complement chromatic number remain fixed. These results demonstrate 
the additional distinguishing power of the Lov\'{a}sz number in this 
setting, although equality of Lov\'{a}sz numbers does not, in general, 
imply graph isomorphism.

The shared spectra also imply equality of additional integer-valued
combinatorial invariants, as follows from standard spectral identities;
see, e.g., \cite{CvetkovicRowlinsonSimic2010,BrouwerHaemers2012}
for background. In particular, adjacency cospectrality implies that
the two graphs have the same total number of closed walks of any
given length, and hence the same numbers of edges and triangles;
see Section~1.3.3 of \cite{BrouwerHaemers2012}.
Laplacian cospectrality implies equality of the numbers of spanning
trees by the matrix-tree theorem; see Section~1.3.5 of
\cite{BrouwerHaemers2012}. Moreover, the adjacency spectrum,
together with any one of the Laplacian, signless Laplacian, or
Seidel spectra, determines the number of cycles of length four.
Signless Laplacian cospectrality also implies adjacency cospectrality
of the line graphs, and hence equality of their total numbers of
closed walks of any given length; see Section~1.4.5 of
\cite{BrouwerHaemers2012}.
For connected graphs, the shared adjacency, Laplacian, and normalized
Laplacian spectra additionally imply equality of the products of the
vertex degrees, by the ordinary and normalized matrix-tree identities;
see \cite{BrouwerHaemers2012,Chung1997}.
Seidel cospectrality implies equality of the numbers of three-vertex
subsets inducing an odd number of edges, as follows from the third
Seidel spectral moment; see also Section~10.2 of
\cite{BrouwerHaemers2012} for the associated two-graph interpretation.
Furthermore, joint adjacency and Seidel cospectrality is equivalent
to adjacency cospectrality of the two graphs and of their complements;
see \cite{JohnsonNewman1980} and Section~14.1 of
\cite{BrouwerHaemers2012}. It also implies equality of the total
numbers of walks of any given length and of the numbers of induced
three-vertex subgraphs of each isomorphism type.
These equalities follow automatically from the shared spectra,
whereas equality of the independence, clique, chromatic, complement
chromatic, and maximum-cut numbers requires separate verification.

Theorem~4.19 of \cite{Sason2024} states that, for every even integer
$n\geq 14$, there exist connected, irregular, nonisomorphic graphs on
$n$ vertices that are cospectral with respect to the adjacency,
Laplacian, signless Laplacian, and normalized Laplacian matrices, have
equal independence, clique, and chromatic numbers, and have distinct
Lov\'{a}sz $\vartheta$-numbers. That construction is based on splitting
joins and produces graphs of order $2k+12$, where $k\geq 1$ is an
integer. It does not address the chromatic numbers of the complements,
the maximum-cut numbers of the graphs, or the maximum-cut numbers of the
corresponding joins formed from their complements.

The main construction of this paper strengthens that existence result
to every integer $n\geq 11$. It removes the parity restriction, covers
all odd orders at least eleven, and includes the previously untreated
even order $n=12$. In addition to sharing the five spectra and the
independence, clique, and chromatic numbers, the two graphs in each pair
have equal chromatic numbers of their complements and equal maximum-cut
numbers. The construction uses the ordinary joins $G\vee K_{n-10}$ and 
$H\vee K_{n-10}$, where $G$ and $H$ are the $4$-regular, cospectral, 
nonisomorphic graphs on ten vertices introduced by van Dam and Haemers
\cite{vanDamHaemers2003}. The regularity of these seed graphs allows a
single orthogonal similarity to establish cospectrality with respect to
all five matrices after the joins are formed. Moreover, joining a graph
with a complete graph preserves its Lov\'{a}sz number. Consequently, the
two Lov\'{a}sz numbers remain distinct and are individually independent
of $n$.

The verification of this strict $\vartheta$-separation also differs from
that in \cite{Sason2024}. Example~4.18 therein reports numerical values
obtained by semidefinite programming for the relevant seed graphs. Here,
we derive exact analytical formulas for the Lov\'{a}sz numbers of $G$
and $H$ and prove that they are distinct. The preservation of these
values under joins then establishes the separation throughout the
infinite family, without reliance on numerical optimization or tolerance
considerations.
We further determine the cardinality-constrained maximum-cut profiles of
the two seed graphs and of their complements. These profiles yield exact
formulas for the relevant maximum-cut numbers and establish their equality
within each pair, both for $G\vee K_{n-10}$ and $H\vee K_{n-10}$, and for 
the corresponding joins $\overline{G}\vee K_{n-10}$ and $\overline{H} 
\vee K_{n-10}$.

At order \(n=10\), we first present a regular pair satisfying all the
stated properties except irregularity. We then exhibit a connected,
irregular, nonisomorphic pair that is cospectral with respect to all
five matrices and has distinct Lov\'{a}sz numbers. The two graphs have
equal independence, clique, chromatic, complement chromatic, and
maximum-cut numbers, as do their complements.
An exhaustive computer search shows that no pair of
connected, irregular, nonisomorphic graphs on at most nine vertices
has identical adjacency, Laplacian, signless Laplacian, and normalized
Laplacian spectra together with equal values of all these combinatorial
invariants. Hence, ten is the smallest possible order of such a pair.
Together with the construction for \(n\geq 11\), this establishes
existence for every integer \(n\geq 10\) and nonexistence for \(n\leq 9\).
These results are established in Theorems~\ref{theorem: all-orders},
\ref{theorem: graphs of order at most 9},
and~\ref{theorem: minimum-order}, stated and proved in
Section~\ref{section: main results and proofs}.
Section~\ref{section: discussion} provides further discussion of the
two seed graphs used in the proof of Theorem~\ref{theorem: all-orders},
and Section~\ref{section: summary and conclusions} summarizes the
results and conclusions.

\section{Main Results and Proofs}
\label{section: main results and proofs}

Throughout this paper, all graphs are finite, undirected, and simple.
Matrices are denoted by boldface capital letters, and vectors by
underlined lowercase letters.

For a simple graph \(F\) on \(n\) vertices, let
\(\boldsymbol{A}(F)\) and \(\boldsymbol{D}(F)\) denote its adjacency matrix and
diagonal degree matrix, respectively. Its Laplacian and signless Laplacian
matrices are defined, respectively, by
\begin{equation}
    \boldsymbol{L}(F) \coloneqq \boldsymbol{D}(F)-\boldsymbol{A}(F),
    \qquad 
    \boldsymbol{Q}(F) \coloneqq \boldsymbol{D}(F)+\boldsymbol{A}(F).
\end{equation}
If \(F\) has no isolated vertices, its normalized Laplacian matrix is defined by
\begin{align}
\label{eq1: normalized Laplacian matrix}
    \boldsymbol{\mathcal{L}}(F) & \coloneqq
    \boldsymbol{D}(F)^{-1/2} \, \boldsymbol{L}(F) \, \boldsymbol{D}(F)^{-1/2} \\
    \label{eq2: normalized Laplacian matrix}
    & = \boldsymbol{I}_n
    -\boldsymbol{D}(F)^{-1/2} \, \boldsymbol{A}(F) \, \boldsymbol{D}(F)^{-1/2},
\end{align}
where \(\boldsymbol{I}_n\) denotes the identity matrix of order \(n\).

Let \(\boldsymbol{J}_{r\times s}\) denote the \(r\times s\) all-ones matrix.
For simplicity, we write
\(\boldsymbol{J}_n\coloneqq\boldsymbol{J}_{n\times n}\).
The Seidel matrix of \(F\) is defined by
\begin{equation}
\label{eq: Seidel matrix}
    \boldsymbol{S}(F)
    \coloneqq \boldsymbol{J}_n-\boldsymbol{I}_n-2\boldsymbol{A}(F).
\end{equation}
Let \(V(F)\) and \(E(F)\) denote, respectively, the vertex and edge sets of $F$.
For \(u\in V(F)\), its neighborhood (i.e., the set of vertices adjacent to
\(u\)) is denoted by \(N_F(u)\), or simply by \(N(u)\) when the underlying
graph is clear. The join of graphs \(F_1\) and \(F_2\), denoted by \(F_1\vee F_2\), 
is obtained from the disjoint union of \(F_1\) and \(F_2\) by adding every edge 
between their vertex sets.

\subsection{Simultaneous Cospectrality with Distinct Lov\'{a}sz Numbers}
\label{subsection: Simultaneous Cospectrality with Distinct Lovász Numbers}

\begin{theorem} \label{theorem: all-orders}
{\em For every integer \(n\geq 11\), there exists a pair of connected,
irregular, nonisomorphic graphs, each on \(n\) vertices, that are
cospectral with respect to the adjacency, Laplacian, signless Laplacian,
normalized Laplacian, and Seidel matrices, have the same independence number,
clique number, chromatic number, complement chromatic number,
and maximum-cut number, but have distinct Lov\'{a}sz \(\vartheta\)-numbers 
whose values are independent of \(n\). The same conclusion holds for 
\(n=10\), with irregularity replaced by regularity.}
\end{theorem}

\begin{proof}
We first specify two seed graphs on 10~vertices and verify their properties, including
an exact separation of their Lov\'asz numbers. We then extend them to every order $n \geq 11$.

\subsubsection{\emph{The seed graphs}}
\label{subsubsection: The seed graphs}

Let $G$ and $H$ have the common vertex set $V=\{a,b,c,d,e,f,g,h,i,j\}$.
Their vertex neighborhoods are listed in Table~\ref{tab:seed-neighborhoods}. 
\begin{table}[htbp]
\centering
\caption{\centering{Neighborhoods of the vertices in the graphs $G$ and $H$.}}
\label{tab:seed-neighborhoods}
\begin{tabular}{|c|c|c|}
\toprule
Vertex & Neighbors in $G$ & Neighbors in $H$ \\
\midrule
$a$ & $b,d,f,h$ & $b,d,g,h$ \\
$b$ & $a,c,d,e$ & $a,c,d,e$ \\
$c$ & $b,e,g,j$ & $b,e,f,j$ \\
$d$ & $a,b,f,h$ & $a,b,f,h$ \\
$e$ & $b,c,g,j$ & $b,c,g,j$ \\
$f$ & $a,d,i,j$ & $c,d,h,i$ \\
$g$ & $c,e,h,i$ & $a,e,i,j$ \\
$h$ & $a,d,g,i$ & $a,d,f,i$ \\
$i$ & $f,g,h,j$ & $f,g,h,j$ \\
$j$ & $c,e,f,i$ & $c,e,g,i$ \\
\bottomrule
\end{tabular}
\end{table}
\begin{figure}[htbp]
    \centering
    \includegraphics[width=1.00\linewidth]{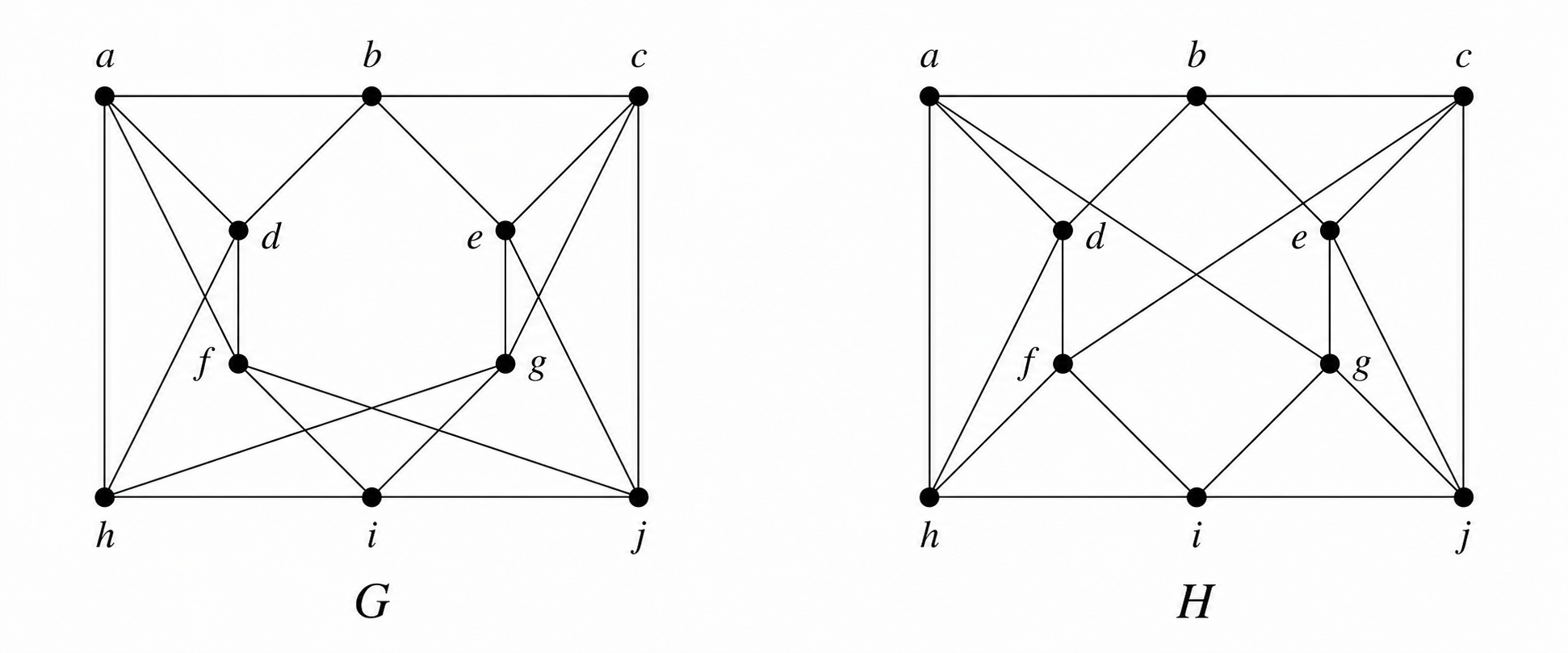}
    \caption{Cospectral, nonisomorphic 4-regular graphs on 10 vertices.}
    \label{fig:G_and_H_regular_graphs}
\end{figure}

\noindent 
These $4$-regular, cospectral, and nonisomorphic graphs were introduced 
by van Dam and Haemers; see Figure~2 of \cite{vanDamHaemers2003}. Unlike 
that figure, Figure~\ref{fig:G_and_H_regular_graphs} includes vertex labels, 
which will be used below. Graph $H$ is constructed from graph $G$ by deleting 
the four edges $\{a,f\}, \{c, g\}, \{h,g\}, \{f,j\}$, and adding the four 
edges $\{a,g\}, \{c,f\}, \{f,h\}, \{g,j\}$. 
We first confirm that $G$ and $H$ are adjacency-cospectral 
and nonisomorphic. 

\begin{itemize}
\item 
For an edge $\{u,v\}$, the number of triangles containing $\{u,v\}$ is
$|N(u)\cap N(v)|$, and the multiset of these numbers is preserved under
graph isomorphisms. In $G$, the edge $\{a,d\}$ is contained in the three
triangles $\{a,b,d\}$, $\{a,d,f\}$, and $\{a,d,h\}$, whereas every edge of 
$H$ is contained in at most two triangles. Hence, $G$ and $H$ are nonisomorphic.

\item 
A direct calculation shows that the adjacency matrices of $G$ and $H$ have 
the same characteristic polynomial:
\begin{align}
 \det\bigl(x \boldsymbol{I}_{10}-\boldsymbol{A}(G)\bigr)
 &=(x-4)(x-1)(x+1)^4(x^2-5)(x^2+x-4) \notag \\
 \label{eq:seed-polynomial}
 &=\det\bigl(x \boldsymbol{I}_{10}-\boldsymbol{A}(H)\bigr).
\end{align}
Thus $G$ and $H$ are adjacency-cospectral. Due to their regularity, the
adjacency-cospectrality of $G$ and $H$ implies their cospectrality also 
with respect to the Laplacian, signless Laplacian, normalized Laplacian, 
and Seidel matrices. Indeed, adjacency-cospectral regular
graphs have the same number of vertices and the same degree.
Let $F$ be a $d$-regular graph on $n$ vertices, with $d>0$, whose
adjacency eigenvalues are
\(
d = \lambda_1 \geq \lambda_2 \geq \ldots \geq \lambda_n.
\)
The identities
\[
\boldsymbol{L}(F)=d\boldsymbol{I}_n-\boldsymbol{A}(F),
\qquad
\boldsymbol{Q}(F)=d\boldsymbol{I}_n+\boldsymbol{A}(F),
\qquad
\boldsymbol{\mathcal{L}}(F)
=\boldsymbol{I}_n-\frac{1}{d}\boldsymbol{A}(F)
\]
show that the corresponding eigenvalues, counted with multiplicity, are
\[
0,\ d-\lambda_2,\ldots,d-\lambda_n
\]
for the Laplacian matrix,
\[
2d,\ d+\lambda_2,\ldots,d+\lambda_n
\]
for the signless Laplacian matrix, and
\[
0,\ 1-\frac{\lambda_2}{d},\ldots,1-\frac{\lambda_n}{d}
\]
for the normalized Laplacian matrix.
Furthermore, the Seidel matrix satisfies \eqref{eq: Seidel matrix}.
Choose the all-ones vector as an adjacency eigenvector corresponding
to $\lambda_1=d$, and choose the remaining eigenvectors orthogonal
to it. Since $\boldsymbol{J}_n$ acts as multiplication by $n$ on the
all-ones vector and as zero on its orthogonal complement, the Seidel
eigenvalues are
\[
n-1-2d,\ -1-2\lambda_2,\ldots,-1-2\lambda_n.
\]
Thus, the adjacency spectrum of a regular graph with positive degree
determines its spectra with respect to all four matrices.
No connectivity assumption is required. The Laplacian, signless
Laplacian, and Seidel formulas also hold when $d=0$.
\end{itemize}

\smallskip 
\noindent 
For a graph $X$, let
\[
 I_X(z) \coloneqq \sum_{k\geq 0} i_k(X) \, z^k,
 \qquad
 C_X(z) \coloneqq \sum_{k\geq 0} c_k(X) \, z^k,
\]
where $i_k(X)$ and $c_k(X)$ denote, respectively, the numbers of
independent sets and cliques of cardinality $k$. Direct enumeration from
the neighborhood table (Table~\ref{tab:seed-neighborhoods}) gives, for 
both $X=G$ and $X=H$,
\begin{equation}\label{eq:count-polynomials}
 I_X(z)=1+10z+25z^2+12z^3,
 \qquad
 C_X(z)=1+10z+20z^2+8z^3.
\end{equation}
Consequently,
\begin{equation}\label{eq:seed-alpha-omega}
 \alpha(G)=\alpha(H)=\omega(G)=\omega(H)=3.
\end{equation}
The sets
\[
 \{a,c,i\},\quad\{b,f,g\},\quad\{d,e\},\quad\{h,j\}
\]
are independent in both graphs and give proper $4$-colorings. Since each
independent set has at most three vertices, three color classes cannot
cover ten vertices. Hence
\begin{equation}
\label{eq:seed-chi}
 \chi(G)=\chi(H)=4.
\end{equation}

\medskip 

The complements also have equal chromatic numbers. By \eqref{eq:seed-alpha-omega}, 
since $\alpha(\overline G)=\omega(G)=3$ and $\alpha(\overline H)=\omega(H)=3$, 
we have
\[
 \chi(\overline G), \, \chi(\overline H)
 \geq\left\lceil\tfrac{10}{3}\right\rceil=4.
\]
The following clique partitions of the respective seed graphs yield
proper $4$-colorings of their complements:
\[
\begin{aligned}
 G:&\quad \{g,h,i\},\ \{c,e,j\},\ \{a,d,f\},\ \{b\},\\
 H:&\quad \{g,i,j\},\ \{d,f,h\},\ \{b,c,e\},\ \{a\}.
\end{aligned}
\]
Consequently,
\begin{equation}
\label{eq:seed-chi-complement}
 \chi(\overline G)=\chi(\overline H)=4.
\end{equation}

\begin{definition}
\label{definition: cardinality-constrained max-cut}
{\em Let \(G=(V,E)\) be a finite simple graph of order \(n\). For every
integer \(k\) with \(0\leq k\leq n\), define the cardinality-constrained 
maximum-cut 
\begin{equation}
\label{eq: cardinality-constrained max-cut}
\mathrm{mc}_k(G)
\coloneqq
\max_{\substack{S\subseteq V\\ |S|=k}} e_G(S,V\setminus S),
\end{equation}
where \(e_G(S,V\setminus S)\) denotes the number of edges having one
endpoint in \(S\) and the other in \(V\setminus S\). The maximum-cut 
number of \(G\) is given by 
\begin{equation}
\label{eq: max-cut}
\mathrm{mc}(G)
\coloneqq \max_{S\subseteq V} e_G(S,V\setminus S)
= \max_{0\leq k\leq n} \mathrm{mc}_k(G).
\end{equation}}
\end{definition}

\begin{proposition}
\label{proposition: max-cut of regular G, H and complements}
{\em For the graphs \(G\) and \(H\) in Figure~\ref{fig:G_and_H_regular_graphs},
\begin{align}
\label{eq1: 16.09.26}
\mathrm{mc}_k(G)=\mathrm{mc}_k(H), \qquad \mathrm{mc}_k(\overline{G})=\mathrm{mc}_k(\overline{H})
\qquad (0\leq k\leq 10),
\end{align}
and their values are given by 
\begin{align}
\label{eq5: 16.09.26}
\begin{array}{c|ccccccccccc}
k 
 & 0 & 1 & 2 & 3 & 4 & 5 & 6 & 7 & 8 & 9 & 10\\
\hline
\mathrm{mc}_k(G)
 & 0 & 4 & 8 & 12 & 14 & \mathbf{16} & 14 & 12 & 8 & 4 & 0\\
\mathrm{mc}_k(H)
 & 0 & 4 & 8 & 12 & 14 & \mathbf{16} & 14 &12 & 8 & 4 & 0\\
\mathrm{mc}_k(\overline{G})
 & 0 & 5 & 10 & 15 & 18 & \mathbf{19} & 18 & 15 & 10 & 5 & 0\\
\mathrm{mc}_k(\overline{H})
 & 0 & 5 & 10 & 15 & 18 & \mathbf{19} & 18 & 15 & 10 & 5 & \, 0.
\end{array}
\end{align}
Consequently, 
\begin{align}
\label{eq6: 16.09.26}
\mathrm{mc}(G)=\mathrm{mc}(H)= 16, \qquad 
\mathrm{mc}(\overline{G}) = \mathrm{mc}(\overline{H}) = 19. 
\end{align}}
\end{proposition}

\begin{proof}
Both \(G\) and \(H\) are \(4\)-regular graphs on the same number of vertices (10). Thus,
for every \(X\in\{G,H\}\) and every \(S \subseteq V(X)\),
\begin{equation}
\label{eq1: 1.10.26}
e_X(S, V(X) \setminus S) = 4 |S|-2 e_X(S),
\end{equation}
where \(e_X(S)\) is the number of edges in the subgraph induced by
\(S\). By \eqref{eq:seed-alpha-omega}, 
\(
\alpha(G)=\alpha(H)=3.
\)
Consequently, for \(0\leq k\leq 3\), one may choose an independent set
\(S\) of cardinality \(k\), and hence
\[
\mathrm{mc}_k(G)=\mathrm{mc}_k(H)=4k.
\]

\smallskip
\noindent 
If \(|S|=4\), then \(e_X(S)\geq 1\), because \(\alpha(X)=3\). Therefore, by \eqref{eq1: 1.10.26}
\[
e_X(S,V(X)\setminus S)\leq 4\cdot 4-2=14.
\]
This bound is attained in \(G\) by
\[
S=\{a,b,g,j\},
\]
which induces only the edge \( \{a,b\} \), and in \(H\) by
\[
S=\{a,b,f,j\},
\]
which also induces only the edge \( \{a,b\} \). It follows that
\(
\mathrm{mc}_4(G)=\mathrm{mc}_4(H)=14.
\)

\smallskip
\noindent 
If \(|S|=5\), then \(e_X(S)\geq 2\). Indeed, if the subgraph induced by
\(S\) had at most one edge, then \(S\) would contain an independent set
of cardinality at least \(4\), contradicting \(\alpha(X)=3\). Hence, by 
\eqref{eq1: 1.10.26}
\[
e_X(S,V(X)\setminus S)\leq 4\cdot 5-2\cdot 2=16.
\]
For both graphs, the set \( S=\{a,c,d,e,i\} \)
induces exactly the two edges \( \{a,d \} \) and \( \{c,e\} \). Consequently,
\[
\mathrm{mc}_5(G)=\mathrm{mc}_5(H)=16.
\]

We next consider the complementary graphs. Since \(G\) and \(H\) are
\(4\)-regular on \(10\) vertices, both \(\overline{G}\) and
\(\overline{H}\) are \(5\)-regular. Thus, for
\(Y\in\{\overline{G},\overline{H}\}\) and \(S\subseteq V(Y)\),
\begin{equation}
\label{eq2: 1.10.26}
e_Y(S,V(Y)\setminus S) = 5 \, |S|-2 e_Y(S).
\end{equation}
By \eqref{eq:seed-alpha-omega}, 
\[
\alpha(\overline{G})=\alpha(\overline{H})=3.
\]
Therefore, for \(0\leq k\leq 3\),
\[
\mathrm{mc}_k(\overline{G})
= \mathrm{mc}_k(\overline{H})
= 5k.
\]

\smallskip
\noindent 
If \(|S|=4\), then \(e_Y(S)\geq 1\), and hence by \eqref{eq2: 1.10.26}
\[
e_Y(S,V(Y)\setminus S)
\leq 5\cdot 4-2=18.
\]
For \(\overline{G}\), the set \( S=\{a,b,d,f\} \)
induces exactly one edge, namely the edge \( \{b,f\} \). For \(\overline{H}\), the set
\( S=\{a,b,d,h\} \) induces exactly one edge, namely \( \{b,h\} \). Therefore,
\[
\mathrm{mc}_4(\overline{G})
= \mathrm{mc}_4(\overline{H})
= 18.
\]

Direct inspection of the adjacency lists for $G$ and $H$ (Table~\ref{tab:seed-neighborhoods})
shows that every set of five vertices in either \(\overline{G}\) or \(\overline{H}\) induces
at least three edges. Hence, if \(|S|=5\), then by \eqref{eq2: 1.10.26}
\[
e_Y(S,V(Y)\setminus S)
\leq 5\cdot 5-2\cdot 3=19.
\]
For both complementary graphs, the set \( S=\{a,b,d,f,h\} \)
induces exactly three edges. More precisely, these edges are
\( \{b,f\},\ \{b,h\},\ \{f,h\} \) in \(\overline{G}\), and
\( \{a,f\},\ \{b,f\},\ \{b,h\} \) in \(\overline{H}\). Consequently,
\[
\mathrm{mc}_5(\overline{G})
= \mathrm{mc}_5(\overline{H})
= 19.
\]

By symmetry, a set and its complement determine the same cut. Therefore, for every
graph \(F\) of order \(n\), and every integer \(0 \leq k \leq n\),
\[
\mathrm{mc}_k(F)=\mathrm{mc}_{n-k}(F).
\]
Applying this identity to \(G\), \(H\), \(\overline{G}\), and
\(\overline{H}\) with \( n = 10 \) gives all the asserted values for \(6\leq k\leq 10\). This completes the proof
of \eqref{eq1: 16.09.26}--\eqref{eq5: 16.09.26}.
Finally, combining \eqref{eq: max-cut} and \eqref{eq5: 16.09.26} yields \eqref{eq6: 16.09.26}.
\end{proof}

\medskip
\subsubsection{An exact separation of the Lov\'{a}sz theta numbers}
\label{subsubsection: An exact separation of the Lov\'asz numbers}

\begin{proposition}
\label{proposition: Lovasz number of G}
{\em For the graph $G$ specified in Figure~\ref{fig:G_and_H_regular_graphs},
\begin{equation}
\label{eq: Lovasz number of G}
\vartheta(G)=1+\sqrt{5} = 3.2360\ldots.
\end{equation}}
\end{proposition}

\begin{proof}
We prove \eqref{eq: Lovasz number of G} in four steps, deriving
matching upper and lower bounds on \(\vartheta(G)\).

\medskip
\noindent\textbf{Step 1: The Gram-matrix formulation of the Lov\'{a}sz theta number.}
Recall the semidefinite characterization of $\vartheta(F)$ for a simple graph \(F\) on \(n\) vertices:
\begin{equation}\label{eq:theta-sdp}
\vartheta(F)=
\max\bigl\{
\langle \boldsymbol{J}_n, \boldsymbol{X} \rangle:
\boldsymbol{X} \succeq \boldsymbol{0} ,\quad \tr(\boldsymbol{X})=1,\quad
X_{v,w}=0 \text{ for all } \{v,w\}\in E(F)
\bigr\},
\end{equation}
where \(\boldsymbol{J}_n\) denotes the all-ones \(n \times n\) matrix. Thus, the maximization in \eqref{eq:theta-sdp} 
is performed over all real positive semidefinite matrices \(\boldsymbol{X}\) of order \(n\), whose entries 
corresponding to the edges of \(F\) are zero.

To obtain an equivalent Gram-matrix formulation, recall that every real positive semidefinite matrix is a 
Gram matrix. Hence, for every positive semidefinite matrix \(\boldsymbol{X} \succeq \boldsymbol{0} \), there 
exists a family of real vectors \(\{\underline{x}_v\}_{v\in V(F)}\) such that
$$
X_{v,w}=\langle \underline{x}_v, \underline{x}_w\rangle,
\qquad v,w \in V(F).
$$
Under this representation, the edge constraints become
\begin{equation}
\label{eq: orthogonality}
\langle \underline{x}_v, \underline{x}_w\rangle=0
\qquad\text{for all }\{v,w\}\in E(F),
\end{equation}
and the trace constraint is equivalent to
\begin{equation}
\label{eq: trace constraint}
\tr(\boldsymbol{X})
=\sum_{v\in V(F)} X_{v,v}
=\sum_{v\in V(F)}\|\underline{x}_v\|^2 =1.
\end{equation}
Moreover,
\begin{align}
\langle \boldsymbol{J}_n, \boldsymbol{X} \rangle
&=\sum_{v,w\in V(F)} X_{v,w} \nonumber \\
&=\sum_{v,w\in V(F)}\langle \underline{x}_v, \underline{x}_w\rangle \nonumber \\
\label{eq: inner product of Jn and X}
&=\left\langle\sum_{v\in V(F)} \underline{x}_v,
\sum_{w\in V(F)} \underline{x}_w\right\rangle 
=\left\|\sum_{v\in V(F)} \underline{x}_v\right\|^2.
\end{align}
Therefore, \eqref{eq:theta-sdp} is equivalent to maximizing
\(\bigl\|\sum_v \underline{x}_v\bigr\|^2\) over all such feasible vector families satisfying
\(\sum_v\|\underline{x}_v\|^2=1\). Finally, since simultaneously multiplying all the vectors 
by a nonzero scalar does not affect the ratio
$$
\frac{\left\|\sum_{v\in V(F)} \underline{x}_v\right\|^2 \rule[-0.8em]{0pt}{0pt}}
     {\rule{0pt}{1.2em} \sum_{v\in V(F)} \|\underline{x}_v\|^2},
$$
the normalization may be incorporated into the denominator. Equivalently,
\begin{equation}\label{eq:theta-gram}
\vartheta(F)=
\max \frac{\left\|\sum_{v\in V(F)} \underline{x}_v\right\|^2 \rule[-0.8em]{0pt}{0pt}}
     {\rule{0pt}{1.2em} \sum_{v\in V(F)} \|\underline{x}_v\|^2},
\end{equation}
where the maximum is over all families of real vectors that are not all zero and satisfy \eqref{eq: orthogonality}.
Conversely, every such family determines a feasible matrix in \eqref{eq:theta-sdp} by setting
\begin{equation}
\label{eq: normalization}
X_{v,w} = \frac{\langle \underline{x}_v, \underline{x}_w\rangle}{\rule{0pt}{1.2em} \sum_{u\in V(F)}\| \underline{x}_u\|^2}.
\end{equation}

\medskip
\noindent\textbf{Step 2: A partition inequality for the Lov\'{a}sz theta number.}
We next record a partition inequality that follows from standard
properties of the Lov\'{a}sz theta function and include a proof
for completeness. For related properties, see
\cite{Knuth94,Lovasz1979}, Chapter~11 of \cite{Lovasz19},
and Section~2.5 of \cite{Sason2024}.
\begin{lemma}
\label{lemma: partition inequality}
{\em Let \( F \) be a finite simple graph, and let $S$ and $T$ be nonempty 
subsets forming a partition of the vertex set of $F$. Then  
\begin{equation} \label{eq:partition-bound}
\vartheta(F) \leq \vartheta(F[S]) + \vartheta(F[T]),
\end{equation}
where $F[S]$ and $F[T]$ denote, respectively, the induced subgraphs of $F$ on 
$S$ and $T$.} 
\end{lemma}
\begin{proof}
The partition inequality \eqref{eq:partition-bound} follows from two 
standard properties of the Lov\'{a}sz theta function: it is monotonically 
nonincreasing under the addition of edges and additive under disjoint 
unions; see \cite[Section~18]{Knuth94} for the latter property. 
The graph \(F\) is obtained from the disjoint union \(F[S] \sqcup F[T]\) 
by adding all edges of \(F\) with one endpoint in \(S\) and the other in \(T\);
here and throughout, \(\sqcup\) denotes disjoint union of graphs, and \(\dot{\cup}\) 
denotes disjoint union of sets. Therefore,
\[
\vartheta(F)
\leq \vartheta\bigl(F[S] \sqcup F[T]\bigr)
= \vartheta(F[S])+\vartheta(F[T]).
\]
\end{proof}

\medskip
\noindent\textbf{Step 3: The upper bound $\vartheta(G)\leq 1+\sqrt{5}$.} 
For an upper bound on $\vartheta(G)$, consider the following partition of $V(G)$:
\begin{equation}
\label{eq:cycle-classes}
S=\{a,b,c,d,e,f,j\},
\qquad T=\{g,h,i\}.
\end{equation}
Two adjacent vertices \(u\) and \(v\) are called
\emph{adjacent true twins} if they have the same closed neighborhood,
that is, \(N[u]=N[v]\). Equivalently,
\begin{equation}
\label{eq: adjacent true twins}
N(u)\setminus\{v\}=N(v)\setminus\{u\}.
\end{equation}
In the induced subgraph $G[S]$, the vertices $a$ and $d$ are
adjacent true twins, since
\begin{equation}
\label{eq1: adjacent true twins}
N_{G[S]}(a)\setminus\{d\} = N_{G[S]}(d)\setminus\{a\} = \{b,f\}.
\end{equation}
Likewise, $c$ and $e$ are adjacent true twins in \(G[S\setminus\{d\}]\), since
\begin{equation}
\label{eq2: adjacent true twins}
N_{G[S\setminus\{d\}]}(c)\setminus\{e\} = N_{G[S\setminus\{d\}]}(e)\setminus\{c\} = \{b,j\}.
\end{equation}
The following lemma provides the adjacent-true-twin reduction used
below. We include a short, self-contained proof for completeness,
as we have not found an explicit statement of this property in
the references cited in the introductory paragraph preceding
Lemma~\ref{lemma: partition inequality}.
\begin{lemma}
\label{lemma: adjacent true twins}
{\em Let \(u\) and \(v\) be adjacent true twins in a graph \(F\). Then
\begin{equation}
\label{eq3: 1.10.26}
\vartheta(F-v)=\vartheta(F).
\end{equation}
In words, deleting one vertex from a pair of adjacent true twins does not
affect the Lov\'{a}sz theta number of the graph.}
\end{lemma}
\begin{proof}
Since \(F-v\) is an induced subgraph of \(F\), monotonicity of the
Lov\'{a}sz theta number under taking induced subgraphs gives
\(
\vartheta(F-v)\leq \vartheta(F).
\)
Indeed, any feasible matrix for the semidefinite program defining
$\vartheta(F-v)$ in \eqref{eq:theta-sdp} can be extended to the vertex set of 
$F$ by inserting a zero row and a zero column corresponding to $v$. The resulting 
matrix remains positive semidefinite, has trace one, satisfies all edge
constraints for $F$, and has the same objective value.

For the reverse inequality, consider a feasible family
\(\{\underline{x}_w\}_{w\in V(F)}\) in the Gram formulation
\eqref{eq:theta-gram}. Since \(\{u,v\}\in E(F)\), the Gram 
constraints in \eqref{eq: orthogonality} give
\(
\langle \underline{x}_u,\underline{x}_v\rangle=0.
\)
After deleting \(v\), define
\begin{equation}
\label{eq: y-vectors}
\underline{y}_u=\underline{x}_u+\underline{x}_v,
\qquad
\underline{y}_w=\underline{x}_w
\quad\text{for }w\in V(F)\setminus\{u,v\}.
\end{equation}
We first verify that the family in \eqref{eq: y-vectors}
satisfies the orthogonality constraints for \(F-v\). 
If \(w\in V(F)\setminus\{u,v\}\) is adjacent to \(u\) in \(F\),
then, since \(u\) and \(v\) are adjacent true twins in \(F\),
\(w\) is also adjacent to \(v\) in \(F\). Consequently,
\(
\langle \underline{y}_u,\underline{y}_w\rangle
= \langle \underline{x}_u,\underline{x}_w\rangle + \langle \underline{x}_v,\underline{x}_w\rangle
=0.
\)
All other orthogonality constraints in \eqref{eq: orthogonality} are left unchanged. Moreover, by \eqref{eq: y-vectors}
\[
\sum_{w\in V(F-v)}\underline{y}_w = \sum_{w\in V(F)}\underline{x}_w,
\]
so the numerator in \eqref{eq:theta-gram} is left unchanged. Since
\(\underline{x}_u\perp\underline{x}_v\), we also have
\[
\|\underline{y}_u\|^2 = \|\underline{x}_u\|^2+\|\underline{x}_v\|^2.
\]
Thus,
\[
\sum_{w\in V(F-v)}\|\underline{y}_w\|^2 = \sum_{w\in V(F)}\|\underline{x}_w\|^2,
\]
and the normalization in \eqref{eq: trace constraint} is preserved. 
Hence every attainable objective value in the Gram formulation
for \(F\) is also attainable for \(F-v\), which yields
\[
\vartheta(F)\leq\vartheta(F-v).
\]
Combining the two inequalities proves equality \eqref{eq3: 1.10.26}.
\end{proof}

Consequently, applying Lemma~\ref{lemma: adjacent true twins}
twice, first to delete vertex~\(d\) and then to delete vertex~\(e\), 
using \eqref{eq1: adjacent true twins} and
\eqref{eq2: adjacent true twins}, respectively, gives
\[
\vartheta\bigl(G[S]\bigr) = \vartheta\bigl(G[S\setminus\{d,e\}]\bigr).
\]
The reduced graph \( G[S\setminus\{d,e\}]\) is the cycle $a,b,c,j,f,a$, 
which is a copy of the pentagon \( C_5 \). Hence, 
\begin{equation}\label{eq:compression}
\vartheta(G[S]) = \vartheta(C_5) = \sqrt{5}.
\end{equation}
Also, $G[T]=K_3$ and $\vartheta(K_3)=1$. Consequently,
by Lemma~\ref{lemma: partition inequality}
\begin{equation} \label{eq: G-upper}
\vartheta(G)\leq \vartheta(G[S])+\vartheta(G[T])
=1+\sqrt{5}.
\end{equation}

\medskip
\noindent\textbf{Step 4: The matching lower bound
$\vartheta(G)\geq 1+\sqrt{5}$.}
Set
\begin{equation}
\label{eq: r,s}
r \coloneqq \tfrac12 \, (\sqrt{5}-1),
\qquad s \coloneqq r^3=\sqrt{5}-2.
\end{equation}
Consider the symmetric matrix
\begin{align}
\label{eq: matrix B}
\boldsymbol{B}=
\begin{pmatrix}
1&0&r&0&r\\
0&1&0&r&r\\
r&0&1&r&0\\
0&r&r&1&0\\
r&r&0&0&1
\end{pmatrix},
\end{align}
whose rows and columns are indexed, in this order, by vertices
$a,b,c,f,g$. Its eigenvalues (including multiplicities) are
\(
\sqrt{5},
\tfrac12 (5-\sqrt{5}),
\tfrac12 (5-\sqrt{5}),
0, 0.
\)
Hence, $\boldsymbol{B}$ is positive semidefinite and has rank $3$. It is
therefore the Gram matrix of vectors
\(
\underline{u}_a, \underline{u}_b, \underline{u}_c,
\underline{u}_f, \underline{u}_g \in\mathbb{R}^3.
\)
Since all the diagonal entries of $\boldsymbol{B}$ are equal to $1$, all these
vectors have unit norm. By the definition of the Gram matrix,
\begin{equation}
\label{eq4: 1.10.26}
\langle\underline{u}_x,\underline{u}_y\rangle=B_{x,y},
\qquad x,y\in\{a,b,c,f,g\}.
\end{equation}

\noindent 
Define
\begin{equation}
\label{eq5: 1.10.26}
\underline{u}_d=\underline{u}_e=\underline{0},
\qquad
\underline{u}_h=\underline{u}_f,
\qquad
\underline{u}_j=\underline{u}_g,
\qquad
\underline{u}_i
=s(\underline{u}_a+\underline{u}_c-\underline{u}_b).
\end{equation}
For every edge $\{x,y\} \in E(G)$ with $x,y\neq i$, the above definitions,
together with the zero entries of $\boldsymbol{B}$, imply that
\[
\langle\underline{u}_x,\underline{u}_y\rangle=0.
\]
Indeed, any constraint involving $d$ or $e$ holds automatically
because $\underline{u}_d=\underline{u}_e=\underline{0}$.
The remaining constraints reduce, using
$\underline{u}_h=\underline{u}_f$ and
$\underline{u}_j=\underline{u}_g$, to inner products represented
by zero entries of $\boldsymbol{B}$.
It remains to verify the constraints corresponding to the edges
incident with $i$. Since the neighbors of $i$ are $f,g,h,j$, we have
by \eqref{eq5: 1.10.26}
\[
\langle\underline{u}_i,\underline{u}_f\rangle = s(0+r-r) = 0,
\qquad
\langle\underline{u}_i,\underline{u}_g\rangle = s(r+0-r) = 0.
\]
The corresponding inner products with $\underline{u}_h$ and
$\underline{u}_j$ also vanish because
$\underline{u}_h=\underline{u}_f$ and
$\underline{u}_j=\underline{u}_g$. Hence, all edge-orthogonality
constraints in \eqref{eq: orthogonality} are satisfied, and this family 
is feasible for~\eqref{eq:theta-gram}. 
Let 
\begin{equation}
\label{eq6: 1.10.26}
\underline{w}
=\underline{u}_a+\underline{u}_b+\underline{u}_c
 +2\underline{u}_f+2\underline{u}_g,
\qquad
\underline{z}
=\underline{u}_a+\underline{u}_c-\underline{u}_b.
\end{equation}
It follows directly from \eqref{eq: r,s}, \eqref{eq: matrix B}, 
\eqref{eq4: 1.10.26}, and \eqref{eq6: 1.10.26} that
\begin{equation}
\label{eq:wz}
\begin{aligned}
\|\underline{w}\|^2
   &=11+18r=2+9\sqrt{5},\\
\langle\underline{w},\underline{z}\rangle
   &=1+2r=\sqrt{5},\\
\|\underline{z}\|^2
   &=3+2r=2+\sqrt{5}.
\end{aligned}
\end{equation}
By \eqref{eq: r,s}, we have
\begin{align}
\label{eq2: s,r}
s(3+2r)=(\sqrt{5}-2)(\sqrt{5}+2)=1.
\end{align}
Moreover, since $\underline{u}_i=s\underline{z}$, it follows from
\eqref{eq: r,s}, \eqref{eq:wz} and \eqref{eq2: s,r} that
\begin{align}
\label{eq:norm-ui}
\|\underline{u}_i\|^2
=s^2\|\underline{z}\|^2
=s^2(3+2r)
=s
=\sqrt{5}-2.
\end{align}
The seven vectors
\(
\underline{u}_a,\underline{u}_b,\underline{u}_c,
\underline{u}_f,\underline{u}_g,
\underline{u}_h,\underline{u}_j
\)
have unit norms and therefore contribute $7$ to the sum of the
squared norms, whereas
$\underline{u}_d=\underline{u}_e=\underline{0}$. Consequently,
by \eqref{eq:norm-ui}
\begin{equation}
\label{eq7: 16.09.26}
\sum_{v\in V(G)}\|\underline{u}_v\|^2
=7+\|\underline{u}_i\|^2 =5+\sqrt{5}.
\end{equation}
Furthermore, the definitions of the vectors give
\[
\sum_{v\in V(G)}\underline{u}_v
=\underline{w}+\underline{u}_i
=\underline{w}+s\underline{z}.
\]
Thus, by \eqref{eq: r,s}, \eqref{eq:wz}, and \eqref{eq:norm-ui},
\begin{equation}
\label{eq8: 16.09.26}
\left\|\sum_{v\in V(G)}\underline{u}_v\right\|^2 
=\|\underline{w}+s\underline{z}\|^2 
=\|\underline{w}\|^2
  +2s \langle \underline{w},\underline{z} \rangle
  +\|\underline{u}_i\|^2  
=10+6\sqrt{5}.
\end{equation}
Substituting this feasible family into \eqref{eq:theta-gram} and   
relying on \eqref{eq7: 16.09.26} and \eqref{eq8: 16.09.26} yields 
\begin{align}
\label{eq:G-lower}
\vartheta(G) \geq
\frac{\left\| \sum_{v\in V(G)} \underline{u}_v \right\|^2}
     {\rule{0pt}{1.2em} \sum_{v\in V(G)} \|\underline{u}_v\|^2} 
=\frac{10+6\sqrt{5}}{5+\sqrt{5}} 
=1+\sqrt{5}.
\end{align}
Finally, combining the matching upper and lower bounds in \eqref{eq: G-upper} 
and \eqref{eq:G-lower} gives \eqref{eq: Lovasz number of G}.
\end{proof}

\begin{proposition}
\label{proposition: Lovasz number of H}
{\em For the graph $H$ specified in Figure~\ref{fig:G_and_H_regular_graphs},
\begin{equation}
\label{eq: Lovasz number of H}
\vartheta(H) =2+2\rho^\ast+\frac{1-2\rho^\ast}{1-\rho^\ast-(\rho^\ast)^2},
\end{equation}
where
\begin{equation}
\label{eq: rho}
\rho^\ast = \tfrac{1}{2} \left( \sqrt{z} - \sqrt{7-z-4 z^{-1/2}} -1 \right),
\end{equation}
and 
\begin{equation}
\label{eq: z}
z = \frac{7}{3} + \frac{8}{3} \cos\left( \frac{1}{3} \, \arccos \frac{101}{128} \right).
\end{equation}
Numerically, \( \vartheta(H) = 3.2688\ldots \).}
\end{proposition}

\begin{proof}
We derive, in four steps, matching lower and upper
bounds on \(\vartheta(H)\). Throughout the proof, we use the
Gram-matrix formulation \eqref{eq:theta-gram} of the Lov\'{a}sz theta function.

\medskip
\noindent\textbf{Step 1: A one-parameter family of feasible Gram representations.}
Let $r \in (0, \tfrac12)$ and define
\begin{equation}
\label{eq: u,v,q}
u=\frac{1-2r}{1-r},
\qquad
v=\frac{r}{1-r},
\qquad
q=\frac{1-2r}{1-r-r^2}.
\end{equation}
These quantities are positive and satisfy
\begin{equation}
\label{eq:H-parameter-identities}
u+v=1,  \qquad \frac{r^2}{v}=r-r^2, 
\qquad q\left(1+\frac{1}{u}+r\right)=2.
\end{equation}

\noindent 
Assign to the vertices of \(H\) the vectors in \(\mathbb{R}^4\) listed 
in Table~\ref{tab:vectors assigned to V(H)}.
\begin{table}[htbp]
\centering
\caption{\centering{Vectors in \(\mathbb{R}^4\) assigned to the vertices of \(H\) in Figure~\ref{fig:G_and_H_regular_graphs}.}}
\label{tab:vectors assigned to V(H)}
\begin{tabular}{|c|cccr|}
\toprule
$x \in V(H)$ & \multicolumn{4}{c}{$\underline{z}_x \in \mathbb{R}^4$} \vline \\ 
\midrule \midrule 
$a$ & ($1$, & $\sqrt{u}$,     & $0$,           & $\sqrt{v}$) \\[0.2cm]
$b$ & ($q$, & $-\frac{q}{\sqrt{u}}$,  & $-q\sqrt{r}$,  & $0$) \\[0.2cm]
$c$ & ($1$, & $\sqrt{u}$,     & $0$,           & $-\sqrt{v}$) \\[0.2cm]
$d$ & ($r$, & $0$,            & $\sqrt{r}$,    & $-\frac{r}{\sqrt{v}}$) \\[0.2cm]
$e$ & ($r$, & $0,$            & $\sqrt{r}$,    & $\frac{r}{\sqrt{v}}$) \\[0.2cm]
$f$ & ($1$, & $-\sqrt{u}$,    & $0$,           & $\sqrt{v}$) \\[0.2cm]
$g$ & ($1$, & $-\sqrt{u}$,    & $0$,           & $-\sqrt{v}$) \\[0.2cm]
$h$ & ($r$, & $0$,            & $-\sqrt{r}$,   & $-\frac{r}{\sqrt{v}}$) \\[0.2cm]
$i$ & ($q$, & $\frac{q}{\sqrt{u}}$,   & $q\sqrt{r}$,   & $0$) \\[0.2cm]
$j$ & ($r$, & $0$,            & $-\sqrt{r}$,   & $\frac{r}{\sqrt{v}}$) \\[0.15cm] \hline 
\end{tabular}
\end{table}

\noindent 
To verify the edge-orthogonality constraints in \eqref{eq: orthogonality} for \( F = H \), it is convenient to
group the edges of \(H\) according to their inner products:
\[
\begin{array}{|c|c|}
\toprule 
\text{Edges } \{x,y\} \in E(H) & \langle \underline{z}_x,\underline{z}_y\rangle \\ 
\midrule \midrule 
\{a,b\},\ \{b,c\},\ \{f,i\},\ \{g,i\}
    & q-q=0 \\[1mm]
\{a,d\},\ \{a,h\},\ \{c,e\},\ \{c,j\},\ \{d,f\},\ \{e,g\},\ \{f,h\},\ \{g,j\}
    & r-r=0 \\[1mm]
\{b,d\},\ \{b,e\},\ \{h,i\},\ \{i,j\}
    & qr-qr=0 \\[1mm]
\{a,g\},\ \{c,f\}
    & 1-u-v=0 \\[1mm]
\{d,h\},\ \{e,j\}
    & r^2-r+\frac{r^2}{v}=0 \\[0.1cm] \hline 
\end{array}
\]
where the last two expressions in the right-hand column vanish by
\eqref{eq:H-parameter-identities}. Hence
\[
\langle \underline{z}_x,\underline{z}_y\rangle=0
\qquad
\text{for every }\{x,y\}\in E(H),
\]
so the family of vectors in Table~\ref{tab:vectors assigned to V(H)} is feasible in the Gram formulation
\eqref{eq:theta-gram} of \( \vartheta(H) \).

\medskip
\noindent\textbf{Step 2: Evaluation and optimization of the lower bound.}
By \eqref{eq:H-parameter-identities} and Table~\ref{tab:vectors assigned to V(H)}, 
the squared \(\ell_2\)-norms of the vectors \( \{ \underline{z}_x \}_{x \in V(H)} \) are given by 
\begin{equation}
\label{eq1: 07.09.26}
\begin{aligned}
&\|\underline{z}_a\|^2
=\|\underline{z}_c\|^2
=\|\underline{z}_f\|^2
=\|\underline{z}_g\|^2
=1+u+v=2,\\
&\|\underline{z}_b\|^2
=\|\underline{z}_i\|^2
=q^2\left(1+\frac{1}{u}+r\right)=2q,\\
&\|\underline{z}_d\|^2
=\|\underline{z}_e\|^2
=\|\underline{z}_h\|^2
=\|\underline{z}_j\|^2
=r^2+r+\frac{r^2}{v}=2r.
\end{aligned}
\end{equation}
Moreover, the last three coordinates cancel when all the ten vectors in 
Table~\ref{tab:vectors assigned to V(H)} are summed, and 
\begin{equation}
\label{eq2: 07.09.26}
\sum_{x\in V(H)} \underline{z}_x = (4+2q+4r,0,0,0).
\end{equation}
It therefore follows from \eqref{eq:theta-gram}, \eqref{eq1: 07.09.26}, and \eqref{eq2: 07.09.26} that
\[
\vartheta(H) \geq \frac{(4+2q+4r)^2}{8+4q+8r} = 2+q+2r.
\]
Consequently, by \eqref{eq: u,v,q} and the last inequality, 
\begin{align}
\label{eq:H-one-parameter-bound}
\vartheta(H) \geq F(r) 
\coloneqq 2+2r+\frac{1-2r}{1-r-r^2},
\qquad 0<r<\tfrac12.
\end{align}
To obtain the largest lower bound of this form, we maximize \( F \) 
over the interval \( (0,\tfrac12) \). Differentiation gives
\[
F'(r) = \frac{2r^4+4r^3-4r^2-2r+1}{(r^2+r-1)^2}, \qquad 
F''(r) = \frac{2r(2r^2-3r+3)}{(r^2+r-1)^3}, \qquad 0<r<\tfrac12.
\]
For \( r \in (0,\tfrac12) \), the numerator of \(F''(r)\) is positive,
whereas its denominator is negative. Thus, \(F\) is strictly
concave on this interval. Furthermore,
\[
\lim_{r \to 0^+} F'(r)=1,
\qquad
\lim_{r \to \frac12^-} F'(r)=-6.
\]
Hence \(F\) has a unique maximizer \(\rho^\ast\in(0,\tfrac12)\),
characterized by setting the numerator of $F'$ to zero, which gives 
the polynomial equation 
\begin{equation}
\label{eq:H-rho-polynomial}
2(\rho^\ast)^4+4(\rho^\ast)^3
-4(\rho^\ast)^2-2\rho^\ast+1=0.
\end{equation}
The unique root in \((0,\tfrac12)\) is \( \rho^\ast \) as given in \eqref{eq: rho}, 
whose derivation is provided in Appendix~\ref{app:quartic-equation}.
Taking \(r=\rho^\ast\) in \eqref{eq:H-one-parameter-bound} yields
\begin{equation}
\label{eq:H-lower}
\vartheta(H) \geq
2+2\rho^\ast+\frac{1-2\rho^\ast}{1-\rho^\ast-(\rho^\ast)^2}
=3.2688\ldots.
\end{equation}

The strict separation of the Lov\'{a}sz numbers $\vartheta(G)$ and 
$\vartheta(H)$ can already be established without numerical approximations. 
Indeed, taking \(r=\tfrac13\) in \eqref{eq:H-one-parameter-bound} and 
applying Proposition~\ref{proposition: Lovasz number of G} yields
\[
\vartheta(H) \geq \tfrac{49}{15} >1+\sqrt5 =\vartheta(G).
\]

Although these bounds suffice to establish
\(\vartheta(G)<\vartheta(H)\), we proceed further to determine
\(\vartheta(H)\) exactly. Beyond the separation result, this exact
evaluation highlights the role of algebraic structure in semidefinite
optimization: graph automorphisms, their induced invariant subspaces,
and polynomial identities combine to yield explicit certificates
of optimality. To this end, we show that the lower bound in
\eqref{eq:H-lower} is tight.
In Step~3, we construct the dual certificate using the automorphisms
of \(H\) and the complementary-slackness conditions associated with
the primal Gram representation. In Step~4, we verify its positive
semidefiniteness by exploiting the simultaneous eigenspace decomposition
induced by two commuting involutions, thereby reducing the original
\(10\times10\) dual matrix to four blocks of orders two and three.
Appendix~\ref{app:H-block-diagonalization} provides the details of this
decomposition and the resulting block diagonalization.

\medskip
\noindent\textbf{Step 3: Construction of a dual certificate.}
Set \( \rho=\rho^\ast \) and define
\begin{equation}
\label{eq:H-tau}
\tau = 2+2\rho+\frac{1-2\rho}{1-\rho-\rho^2}.
\end{equation}
Recall from \eqref{eq:H-rho-polynomial} that
\begin{equation}
\label{eq:H-rho-relation}
2\rho^4+4\rho^3-4\rho^2-2\rho+1=0,
\qquad 0<\rho<\tfrac12.
\end{equation}
To obtain a matching upper bound on \( \vartheta(H) \), we  
must show that \( \vartheta(H) \leq \tau \).

We first recall a general dual-certificate criterion for the
semidefinite program in \eqref{eq:theta-sdp}, which follows from
weak duality for semidefinite programming; see
Section~11.2 and Theorem~13.5 of \cite{Lovasz19}.
\begin{lemma}
\label{lemma: dual-certificate criterion}
{\em Let \(F\) be a finite simple graph on \(n\) vertices. 
For \(x,y\in V(F)\), let \(\boldsymbol{E}_{x,y}\) denote the 
\(n\times n\) matrix whose \((x,y)\)-entry is equal to \(1\) and
whose remaining entries are zero. Suppose that there exist a real number 
\(t\) and a weight \( w_{\{x,y\}}\in\mathbb{R} \) for each 
unordered edge \(\{x,y\}\in E(F)\), such that
\begin{equation}
\label{eq:general-dual-certificate}
\boldsymbol{M}_F = t \boldsymbol{I}_n - \boldsymbol{J}_n 
+ \sum_{\{x,y\}\in E(F)} w_{\{x,y\}} ( \boldsymbol{E}_{x,y}+\boldsymbol{E}_{y,x} ) \succeq 0.
\end{equation}
Then
\begin{equation}
\label{eq: UB theta}
\vartheta(F) \leq t.
\end{equation}}
\end{lemma}
\begin{proof}
Let \(\boldsymbol{X}\) be any feasible matrix for
\eqref{eq:theta-sdp}. Since
\(\boldsymbol{M}_F \succeq \boldsymbol{0}\), it admits a unique
positive semidefinite square root \(\boldsymbol{M}_F^{1/2}\).
Moreover, the feasibility of \(\boldsymbol{X}\) implies that
\(\boldsymbol{X} \succeq \boldsymbol{0}\). By the cyclicity of the
trace,
\begin{align}
\langle \boldsymbol{M}_F,\boldsymbol{X} \rangle
&= \tr(\boldsymbol{M}_F\boldsymbol{X}) \nonumber \\
\label{eq7: 1.10.26}
&= \tr\bigl( \boldsymbol{M}_F^{1/2} \boldsymbol{X} \boldsymbol{M}_F^{1/2} \bigr)
\geq 0,
\end{align}
where the last inequality holds since
\(
\boldsymbol{M}_F^{1/2} \boldsymbol{X} \boldsymbol{M}_F^{1/2} \succeq \boldsymbol{0}.
\) 
On the other hand, by \eqref{eq:general-dual-certificate}
\begin{align}
\langle \boldsymbol{M}_F, \boldsymbol{X} \rangle
&= t\tr(\boldsymbol{X})-\langle \boldsymbol{J}_n, \boldsymbol{X} \rangle
+ \sum_{\{x,y\}\in E(F)} w_{\{x,y\}} \langle \boldsymbol{E}_{x,y}
+\boldsymbol{E}_{y,x}, \boldsymbol{X} \rangle \nonumber \\
\label{eq8: 1.10.26}
&= t \tr(\boldsymbol{X})-\langle \boldsymbol{J}_n, \boldsymbol{X} \rangle
+ 2\sum_{\{x,y\}\in E(F)} w_{\{x,y\}} X_{x,y},
\end{align}
where the last equality holds by the symmetry of \( \boldsymbol{X} \).
The feasibility constraints in \eqref{eq:theta-sdp} give \( \tr(\boldsymbol{X})=1 \)
and \( X_{x,y}=0 \) for every \( \{x,y\}\in E(F) \).
Consequently, combining \eqref{eq7: 1.10.26} and \eqref{eq8: 1.10.26} yields 
\[
0\leq\langle \boldsymbol{M}_F, \boldsymbol{X} \rangle
= t-\langle \boldsymbol{J}_n, \boldsymbol{X} \rangle,
\]
so 
\[
\langle \boldsymbol{J}_n, \boldsymbol{X} \rangle \leq t.
\]
Taking the maximum of the left-hand side over all feasible matrices \(\boldsymbol{X}\) 
in \eqref{eq:theta-sdp} gives \eqref{eq: UB theta}.
\end{proof}

We now apply the dual-certificate criterion in Lemma~\ref{lemma: dual-certificate criterion} to 
\(F=H\), where \(H\) is specified in Table~\ref{tab:seed-neighborhoods}, with \(n=10\) and \(t=\tau\). 
Thus, it remains to find real edge weights
\[
w_{\{x,y\}} \in \mathbb{R},  \qquad  \forall \, \{x,y\}\in E(H),
\]
such that
\begin{equation}
\label{eq:H-dual-certificate-specialized}
\boldsymbol{M} = \tau \boldsymbol{I}_{10}-\boldsymbol{J}_{10} 
+ \sum_{\{x,y\}\in E(H)} w_{\{x,y\}} ( \boldsymbol{E}_{x,y}+\boldsymbol{E}_{y,x} ) \succeq \boldsymbol{0}.
\end{equation}
The construction of these weights is described next.

We first record the symmetries that will be used in constructing the
weights. Consider the two permutations
\begin{equation}
\label{eq: permutation pi}
\pi=(a\,c)(d\,e)(f\,g)(h\,j),
\end{equation}
and
\begin{equation}
\label{eq: permutation sigma}
\sigma=(a\,f)(b\,i)(c\,g)(d\,h)(e\,j).
\end{equation}
It can be verified from Table~\ref{tab:seed-neighborhoods} that \(\pi\) 
and \(\sigma\) are automorphisms of \(H\). Moreover, they are commuting 
involutions:
\begin{equation}
\label{eq: commuting involutions}
\pi^2=\sigma^2=\operatorname{id},  \qquad  \pi\sigma=\sigma\pi,
\end{equation}
where \( \mathrm{id} \) denotes the identity permutation on \(V(H)\).
The group generated by \(\pi\) and \(\sigma\) has the following three
orbits on \(V(H)\):
\begin{equation}
\label{eq: three orbits}
\{a,c,f,g\}, \qquad \{b,i\}, \qquad \{d,e,h,j\}.
\end{equation}
Indeed, both permutations preserve each of these sets, and applying
\(\mathrm{id}\), \(\pi\), \(\sigma\), and \(\pi\sigma\) to the
representatives \(a\), \(b\), and \(d\) yields all vertices of their
respective sets.

We next determine suitable edge weights. Partition the edge set of
\(H\) into the following six classes, each invariant under both 
\(\pi\) and \(\sigma\):
\begin{equation}
\label{eq:H-edge-classes}
\begin{aligned}
\mathcal{E}_1 &= \bigl\{ \{a,b\},\{b,c\},\{f,i\},\{g,i\} \bigr\}, \\
\mathcal{E}_2 &= \bigl\{ \{a,d\},\{c,e\},\{f,h\},\{g,j\} \bigr\}, \\
\mathcal{E}_3 &= \bigl\{ \{a,g\},\{c,f\} \bigr\}, \\
\mathcal{E}_4 &= \bigl\{ \{a,h\},\{c,j\},\{d,f\},\{e,g\} \bigr\}, \\
\mathcal{E}_5 &= \bigl\{ \{b,d\},\{b,e\},\{h,i\},\{i,j\} \bigr\}, \\
\mathcal{E}_6 &= \bigl\{ \{d,h\},\{e,j\} \bigr\}.
\end{aligned}
\end{equation}
We assign a common weight \(w_\nu\) to every edge in \(\mathcal{E}_\nu\),
for each \( \nu \in [6] \). Thus, we seek a matrix of the form
\begin{equation}
\label{eq:H-dual-matrix-general}
\boldsymbol{M} = \tau \boldsymbol{I}_{10}-\boldsymbol{J}_{10} 
+ \sum_{\nu=1}^{6} \Biggl\{ w_\nu \sum_{\{x,y\}\in \mathcal{E}_\nu} 
( \boldsymbol{E}_{x,y} + \boldsymbol{E}_{y,x} ) \Biggr\}.
\end{equation}

The weights are determined from the complementary-slackness condition
associated with the primal solution constructed in Steps~1 and~2.
Let \(\boldsymbol{Z}\) be the \(10\times4\) matrix whose row indexed by
\(x\in V(H)\) is \(\underline{z}_x^{\mathsf T}\), where
\(\underline{z}_x\) is the vector in
Table~\ref{tab:vectors assigned to V(H)}, evaluated at \(r=\rho\).
The corresponding feasible primal matrix, whose trace is equal to~1 by 
\eqref{eq:theta-sdp}, is given by 
\[
\boldsymbol{X}^\ast =
\frac{\boldsymbol{Z} \boldsymbol{Z}^{\mathsf T}} {\tr(\boldsymbol{Z} \boldsymbol{Z}^{\mathsf T})}.
\]
By Step~2, its objective value is
\[
\langle \boldsymbol{J}_{10}, \boldsymbol{X}^\ast \rangle = \tau.
\]
Therefore, a dual matrix \(\boldsymbol{M}\) with the same objective value should
satisfy the complementary slackness condition 
\(
\langle \boldsymbol{M}, \boldsymbol{X}^\ast \rangle = 0.
\)
Since \( \boldsymbol{X}^\ast\) is a positive scalar multiple of \( \boldsymbol{Z} \boldsymbol{Z}^{\mathsf T}\),
this condition is equivalent to
\begin{align}
\label{eq1: 18.09.26}
\langle \boldsymbol{M}, \boldsymbol{Z} \boldsymbol{Z}^{\mathsf T} \rangle = 0.
\end{align}
For a positive semidefinite matrix \( \boldsymbol{M}\), this is equivalent to
\begin{equation}
\label{eq:H-complementary-slackness}
\boldsymbol{M} \boldsymbol{Z} = \boldsymbol{0}.
\end{equation}
Indeed, 
\begin{align*}
0 &=\bigl\langle \boldsymbol{M},
   \boldsymbol{Z}\boldsymbol{Z}^{\mathsf T}\bigr\rangle \\
&=\tr\bigl(
   \boldsymbol{M}\boldsymbol{Z}\boldsymbol{Z}^{\mathsf T}\bigr)\\
&=\tr\bigl(
   \boldsymbol{Z}^{\mathsf T}\boldsymbol{M}\boldsymbol{Z}\bigr)\\
&=\tr\bigl(
   \boldsymbol{Z}^{\mathsf T}\boldsymbol{M}^{1/2}
   \boldsymbol{M}^{1/2}\boldsymbol{Z}\bigr)\\
&=\bigl\|\boldsymbol{M}^{1/2}\boldsymbol{Z}\bigr\|_{\mathrm F}^{2},
\end{align*}
where \(\|\cdot\|_{\mathrm F}\) denotes the Frobenius norm. Consequently,
\( \boldsymbol{M}^{1/2}\boldsymbol{Z}=\boldsymbol{0} \), which then yields 
\( \boldsymbol{M}\boldsymbol{Z}
=\boldsymbol{M}^{1/2} \bigl(\boldsymbol{M}^{1/2}\boldsymbol{Z}\bigr)
=\boldsymbol{0} \).
Conversely, if \eqref{eq:H-complementary-slackness} holds, then clearly 
\eqref{eq1: 18.09.26} holds. 

We therefore impose \eqref{eq:H-complementary-slackness} to determine
candidate values of the real weights \(w_1,\ldots,w_6\) on the right-hand 
side of \eqref{eq:H-dual-matrix-general}. The positive semidefiniteness of 
the resulting matrix will be verified independently in Step~4.

To write these equations explicitly, observe from \eqref{eq: u,v,q} and 
\eqref{eq2: 07.09.26} (with $r = \rho$) that
\begin{equation}
\label{eq:H-vector-sum}
\underline{s}
\coloneqq
\sum_{x\in V(H)}\underline{z}_x = (4+2q+4\rho, 0, 0, 0),
\end{equation}
where
\[
q=\frac{1-2\rho}{1-\rho-\rho^2}.
\]
The row of \( \boldsymbol{M} \boldsymbol{Z} \) indexed by \(x \in V(H)\) is
\[
(\boldsymbol{M} \boldsymbol{Z})_x =
\tau \underline{z}_x - \underline{s} + \sum_{y \in N_H(x)} w_{x,y} \underline{z}_y.
\]
Consequently, the condition \( \boldsymbol{M} \boldsymbol{Z}=\boldsymbol{0}\) is 
equivalent to
\begin{equation}
\label{eq:H-vector-slackness}
\tau \underline{z}_x + \sum_{y\in N_H(x)} w_{x,y} \, \underline{z}_y = \underline{s},
\qquad x \in V(H).
\end{equation}
To justify the reduction to one representative from each orbit, define
the orthogonal matrices
\[
\boldsymbol{R}_\pi=\operatorname{diag}(1,1,1,-1),
\qquad
\boldsymbol{R}_\sigma=\operatorname{diag}(1,-1,-1,1).
\]
The vectors in Table~\ref{tab:vectors assigned to V(H)} satisfy
\begin{equation}
\label{eq1: 2.10.26}
\underline{z}_{\pi(x)}=\boldsymbol{R}_\pi \, \underline{z}_x,
\qquad
\underline{z}_{\sigma(x)}=\boldsymbol{R}_\sigma \, \underline{z}_x,
\qquad x\in V(H).
\end{equation}
Moreover, by \eqref{eq:H-vector-sum} and \eqref{eq1: 2.10.26}, we have 
\begin{equation}
\label{eq2: 2.10.26}
\boldsymbol{R}_\pi \, \underline{s}=\underline{s},
\qquad
\boldsymbol{R}_\sigma \, \underline{s}=\underline{s}.
\end{equation}
Because each edge class in \eqref{eq:H-edge-classes} is invariant under 
both automorphisms, the weights satisfy
\begin{equation}
\label{eq3: 2.10.26}
w_{\pi(x),\pi(y)}=w_{x,y},
\qquad
w_{\sigma(x),\sigma(y)}=w_{x,y}, \qquad \forall \, \{x, y\} \in E(H).
\end{equation}
Consequently, the equations in \eqref{eq:H-vector-slackness} are
equivariant under \(\pi\) and \(\sigma\). It is therefore enough to
impose them on one representative from each of the three orbits
in \eqref{eq: three orbits}.
Taking \(a\), \(b\), and \(d\) as representatives and relying on the
rightmost column of Table~\ref{tab:seed-neighborhoods}, we obtain from
\eqref{eq:H-edge-classes} and \eqref{eq:H-vector-slackness} the following:
\begin{align}
\tau\underline{z}_a
+w_1\underline{z}_b
+w_2\underline{z}_d
+w_3\underline{z}_g
+w_4\underline{z}_h
&=
\underline{s},
\label{eq:H-slackness-a}\\
\tau\underline{z}_b
+w_1(\underline{z}_a+\underline{z}_c)
+w_5(\underline{z}_d+\underline{z}_e)
&=
\underline{s},
\label{eq:H-slackness-b}\\
\tau\underline{z}_d
+w_2\underline{z}_a
+w_5\underline{z}_b
+w_4\underline{z}_f
+w_6\underline{z}_h
&=
\underline{s}.
\label{eq:H-slackness-d}
\end{align}
For example, equation \eqref{eq:H-slackness-a} follows from
\(
N_H(a)=\{b,d,g,h\},
\)
where by \eqref{eq:H-edge-classes}
\[ 
\{a,b\} \in \mathcal{E}_1,  \qquad 
\{a,d\} \in \mathcal{E}_2,  \qquad 
\{a,g\} \in \mathcal{E}_3,  \qquad  
\{a,h\} \in \mathcal{E}_4.
\] 
The other two equations in \eqref{eq:H-slackness-b} and \eqref{eq:H-slackness-d} 
follow similarly from \( N_H(b)=\{a,c,d,e\} \) and \( N_H(d)=\{a,b,f,h\} \), respectively,
together with \eqref{eq:H-edge-classes}.

\noindent
Substituting the vectors from
Table~\ref{tab:vectors assigned to V(H)} into
\eqref{eq:H-slackness-a}--\eqref{eq:H-slackness-d} and equating
the four coordinates on both sides of each of these three vector
equations yields twelve scalar linear equations in the six
unknown weights. The identities
\begin{equation}
\label{eq4: 2.10.26}
u+v=1,
\qquad
\frac{\rho^2}{v}=\rho-\rho^2,
\qquad
q\left(1+\frac1u+\rho\right)=2,
\end{equation}
which follow from \eqref{eq:H-parameter-identities} with \(r=\rho\),
together with \eqref{eq:H-tau} and \eqref{eq:H-rho-relation},
simplify this system. Solving the resulting equations and using
\eqref{eq:H-rho-relation} to reduce all powers \(\rho^k\) with
\(k \geq 4\), we obtain
\begin{equation}
\label{eq:H-dual-weights}
\begin{aligned}
w_1&= -\frac{28\rho^3+72\rho^2+3\rho-62}{25},\\[1mm]
w_2&= -\frac{88\rho^3+362\rho^2+38\rho-377}{125},\\[1mm]
w_3&= \frac{188\rho^3+512\rho^2+138\rho-27}{125},\\[1mm]
w_4&= -\frac{134\rho^3+216\rho^2-266\rho-111}{125},\\[1mm]
w_5&= \frac{26\rho^3+74\rho^2-24\rho+21}{25},\\[1mm]
w_6&= -\frac{2(134\rho^3+216\rho^2-266\rho-111)}{125}.
\end{aligned}
\end{equation}
Direct substitution, followed by reduction using
\eqref{eq:H-rho-relation}, confirms that these weights satisfy
all twelve scalar equations.
Thus, the weights in \eqref{eq:H-dual-weights} are obtained by 
imposing complementary slackness on the symmetry-invariant family 
\eqref{eq:H-dual-matrix-general}.
We henceforth let \( \boldsymbol{M} \) be given by 
\eqref{eq:H-dual-matrix-general}, together with the real weights 
\( \{ w_\nu \}_{\nu=1}^6 \) as given by \eqref{eq:H-dual-weights}. 
By construction, \eqref{eq:H-complementary-slackness} holds. 
It remains to verify that \( \boldsymbol{M} \succeq \boldsymbol{0}\), 
which is carried out in Step~4.

\medskip
\noindent\textbf{Step 4: Verifying that \(\boldsymbol{M}\) 
is positive semidefinite.}
Recall the commuting involutive automorphisms \(\pi\) and \(\sigma\) 
(see \eqref{eq: permutation pi}--\eqref{eq: commuting involutions}).
Since each of the six edge classes of \(H\) in \eqref{eq:H-edge-classes}
is invariant under both \(\pi\) and \(\sigma\), and the weights are
constant on each class, the matrix \(\boldsymbol{M}\) in
\eqref{eq:H-dual-matrix-general} commutes with the corresponding
permutation matrices. 
We exploit these symmetries to block diagonalize \(\boldsymbol{M}\);
for the general theory of symmetry reduction in semidefinite programming,
see~\cite{Vallentin2009}.

Let \(\boldsymbol{P}_\pi\) and \(\boldsymbol{P}_\sigma\) denote the 
permutation matrices associated with \(\pi\) and \(\sigma\), respectively. 
Since these matrices are commuting symmetric involutions, \(\mathbb{R}^{10}\) 
decomposes as the orthogonal direct sum of their joint eigenspaces:
\begin{equation}
\label{eq: direct sum}
\mathbb{R}^{10} = W_{++} \oplus W_{+-} \oplus W_{-+} \oplus W_{--},
\end{equation}
where
\begin{equation}
\label{eq: W-subspaces}
W_{\varepsilon,\delta}
= \bigl\{ \underline{v} \in \mathbb{R}^{10}:
\boldsymbol{P}_\pi \, \underline{v} = \varepsilon \underline{v},\
\boldsymbol{P}_\sigma \, \underline{v} = \delta \underline{v} \bigr\},
\qquad \varepsilon,\delta \in \{+1,-1\}.
\end{equation}
Because \(\boldsymbol{M}\) commutes with both \(\boldsymbol{P}_\pi\) 
and \( \boldsymbol{P}_\sigma\), each of these four subspaces is invariant 
under \(\boldsymbol{M}\); that is, \(\boldsymbol{M}\underline{v} 
\in W_{\varepsilon,\delta}\) for every \(\underline{v} \in W_{\varepsilon,\delta}\).
Thus, choosing an orthonormal basis for each joint eigenspace yields
a block diagonal representation of \(\boldsymbol{M}\)
(see Appendix~\ref{app:H-block-diagonalization}).

Let \(\{\underline{e}_x: x \in V(H)\}\) be the standard orthonormal
basis of \(\mathbb{R}^{10}\). For
\(\varepsilon,\delta\in\{+1,-1\}\), define
\begin{align}
\label{eq: A,D}
\begin{aligned}
\underline{u}_{\varepsilon,\delta}
&= \tfrac12 \bigl(
\underline{e}_a +\varepsilon \underline{e}_c
+\delta\underline{e}_f
+\varepsilon \delta \underline{e}_g \bigr), \\
\underline{v}_{\varepsilon,\delta}
&= \tfrac12 \bigl( \underline{e}_d
+\varepsilon \underline{e}_e
+\delta \underline{e}_h
+\varepsilon \delta \underline{e}_j \bigr),
\end{aligned}
\end{align}
and, for \(\delta \in \{+1,-1\}\), define
\begin{equation}
\label{eq: B}
\underline{t}_{+,\delta}
= \tfrac{1}{\sqrt{2}} \bigl( \underline{e}_b + \delta \underline{e}_i \bigr).
\end{equation}
These vectors give the following orthonormal bases:
\begin{align}
\begin{aligned}
W_{++} &= \operatorname{span}
\{\underline{u}_{++},\underline{t}_{++},\underline{v}_{++}\},\\
W_{+-} &= \operatorname{span}
\{\underline{u}_{+-},\underline{t}_{+-},\underline{v}_{+-}\},\\
W_{-+} &= \operatorname{span}
\{\underline{u}_{-+},\underline{v}_{-+}\},\\
W_{--} &= \operatorname{span}
\{\underline{u}_{--},\underline{v}_{--}\}.
\end{aligned}
\end{align}
In particular,
\begin{align}
\label{eq: dimensions W-subspaces}
\dim W_{++}=3,\qquad \dim W_{+-}=3,\qquad \dim W_{-+}=2,\qquad \dim W_{--}=2.
\end{align}
With respect to these ordered bases, the four restrictions of \( \boldsymbol{M} \)
are represented by the following matrices:
\begin{equation}
\label{eq:H-block-pp}
\boldsymbol{M}_{++} =
\begin{pmatrix}
\tau-4+w_3  & \sqrt{2}(w_1-2)  & w_2+w_4-4 \\
\sqrt{2}(w_1-2)  & \tau-2  & \sqrt{2} (w_5-2) \\
w_2+w_4-4  & \sqrt{2}(w_5-2)  & \tau-4+w_6
\end{pmatrix},
\end{equation}
\begin{equation}
\label{eq:H-block-pm}
\boldsymbol{M}_{+-} =
\begin{pmatrix}
\tau-w_3  & \sqrt{2}w_1  & w_2-w_4 \\
\sqrt{2}w_1  & \tau  & \sqrt{2} w_5 \\
w_2-w_4  & \sqrt{2} w_5  &  \tau-w_6
\end{pmatrix},
\end{equation}
\begin{equation}
\label{eq:H-block-mp}
\boldsymbol{M}_{-+} =
\begin{pmatrix}
\tau-w_3  & w_2+w_4\\
w_2+w_4   & \tau+w_6
\end{pmatrix},
\end{equation}
\begin{equation}
\label{eq:H-block-mm}
\boldsymbol{M}_{--} =
\begin{pmatrix}
\tau+w_3 & w_2-w_4\\
w_2-w_4 & \tau-w_6
\end{pmatrix}.
\end{equation}
The details of this block diagonalization and the verification of the
entries in \eqref{eq:H-block-pp}--\eqref{eq:H-block-mm} are provided
in Appendix~\ref{app:H-block-diagonalization}.
Thus, instead of checking positive semidefiniteness of the original
\(10\times10\) matrix, it is enough to check the four matrices in
\eqref{eq:H-block-pp}--\eqref{eq:H-block-mm}.

Substitution of the weights from \eqref{eq:H-dual-weights}, followed
by reduction modulo \eqref{eq:H-rho-relation}, gives
\begin{equation}
\label{eq: rank identities}
\operatorname{rank}(\boldsymbol{M}_{++})=2, \qquad
\operatorname{rank}(\boldsymbol{M}_{+-})=1, \qquad
\operatorname{rank}(\boldsymbol{M}_{-+})=1, \qquad
\operatorname{rank}(\boldsymbol{M}_{--})=2.
\end{equation}
The total rank is therefore
\(
\operatorname{rank}(\boldsymbol{M})=2+1+1+2=6.
\)
This is also consistent with the complementary-slackness relation
\(\boldsymbol{M}\boldsymbol{Z}=\boldsymbol{0}\). Indeed, the four
columns of \(\boldsymbol{Z}\) are linearly independent and therefore
span a four-dimensional subspace of \(\ker\boldsymbol{M}\). Since
\(\operatorname{rank}(\boldsymbol{M})=6\), the rank--nullity theorem
gives
\[
\dim\ker\boldsymbol{M}=4.
\]
Hence the columns of \(\boldsymbol{Z}\) span
\(\ker\boldsymbol{M}\).

Recall that if a real symmetric matrix has rank $r$ and its
first $r$ leading principal minors are strictly positive,
then its leading $r\times r$ principal submatrix is positive
definite by Sylvester's criterion. The rank assumption then
forces the corresponding Schur complement to vanish, so the
full matrix is positive semidefinite.

Applying this observation to the rank identities in
\eqref{eq: rank identities}, positive semidefiniteness of the
four blocks follows if the first two leading principal minors
of each rank-two block and the first leading principal minor
of each rank-one block are strictly positive.
Since the first leading principal minors of
$\boldsymbol{M}_{+-}$ and $\boldsymbol{M}_{-+}$ both equal
$\tau-w_3$ (see \eqref{eq:H-block-pm} and \eqref{eq:H-block-mp}), 
it suffices to verify strict positivity of the following five quantities, 
each of which is a positive scalar multiple of the corresponding 
leading principal minor:
\begin{align}
\label{eq:H-delta-pp-1}
\Delta_{++}^{(1)} &= \frac{4}{125} \left( 288\rho^3+662\rho^2+238\rho-177 \right), \\[0.1cm]
\label{eq:H-delta-pp-2}
\Delta_{++}^{(2)} &= -\frac{64}{625} \left( 816\rho^3-1916\rho^2-1034\rho+561 \right),\\[0.1cm]
\label{eq:H-delta-pm}
\Delta_{+-}^{(1)} = \Delta_{-+}^{(1)} &= -\frac{4}{125} \left( 88\rho^3+362\rho^2+38\rho-377 \right), \\[0.1cm]
\label{eq:H-delta-mm-1}
\Delta_{--}^{(1)} &= \frac{4}{125} \left( 288\rho^3+662\rho^2+238\rho+323 \right), \\[0.1cm]
\label{eq:H-delta-mm-2}
\Delta_{--}^{(2)} &= \frac{64}{15625} \left( 43694\rho^3+144531\rho^2+18294\rho-23476 \right).
\end{align}
More precisely, \(\Delta_{++}^{(1)}\) and
\(\Delta_{++}^{(2)}\) are respectively \(4\) and \(32\) times the
first and second leading principal minors of
\(\boldsymbol{M}_{++}\). Moreover,
\(\Delta_{+-}^{(1)}\) and \(\Delta_{-+}^{(1)}\) are each four times
the first leading principal minor of the corresponding rank-one
block. Finally, \(\Delta_{--}^{(1)}\) is four times the first leading
principal minor of \(\boldsymbol{M}_{--}\), whereas
\(\Delta_{--}^{(2)}\) is \(16\) times its determinant.

For \(\boldsymbol{M}_{++}\), positivity of its first two leading
principal minors, together with the equality 
\(\operatorname{rank}(\boldsymbol{M}_{++})=2\), implies that
\(\boldsymbol{M}_{++}\) is positive semidefinite. For each of the
rank-one blocks \(\boldsymbol{M}_{+-}\) and
\(\boldsymbol{M}_{-+}\), positivity of the displayed diagonal entry
implies positive semidefiniteness. Finally, positivity of the two
leading principal minors of the full-rank \(2 \times 2\) block
\(\boldsymbol{M}_{--}\) implies that this block is positive definite.

\smallskip 
Following \eqref{eq:H-rho-relation}, let
\[
p(x)=2x^4+4x^3-4x^2-2x+1.
\]
A direct calculation gives \( p(0.353)>0 \) and  \( p(0.354)<0.\)
Since \(\rho\) is the unique root of \(p\) in \((0,\tfrac12)\), 
it follows that \( 0.353 < \rho < 0.354 \).
The positivity checks can now be carried out using only rational
interval arithmetic. On this interval, the expressions in
\eqref{eq:H-delta-pp-1}--\eqref{eq:H-delta-mm-2} satisfy
\[
\begin{aligned}
&\Delta_{++}^{(1)}
>0.0695, \qquad 
\Delta_{++}^{(2)}
>0.7024, \qquad 
&\Delta_{+-}^{(1)}
=\Delta_{-+}^{(1)}
>10.0569,\\
&\Delta_{--}^{(1)}
>16.0695, \qquad 
\Delta_{--}^{(2)}
>11.9341.
\end{aligned}
\]
For $\Delta_{++}^{(2)}$, we additionally use that its derivative
is strictly positive on $[0.353,0.354]$, so its value at $\rho$
is strictly greater than its value at $0.353$.
Since all the scaling factors described above are positive, these
inequalities establish the required positivity of the corresponding
leading principal minors. Together with the rank identities, this gives
\[
\boldsymbol{M}_{++} \succeq \boldsymbol{0}, \qquad
\boldsymbol{M}_{+-} \succeq \boldsymbol{0}, \qquad
\boldsymbol{M}_{-+} \succeq \boldsymbol{0}, \qquad
\boldsymbol{M}_{--} \succeq \boldsymbol{0}.
\]
Since \(\mathbb{R}^{10}\) is the orthogonal direct sum of the four
corresponding invariant subspaces, it follows that
\(
\boldsymbol{M} \succeq \boldsymbol{0}.
\)
The matrix \( \boldsymbol{M} \) is therefore a feasible dual certificate with
objective value \(\tau\). Hence,
\[
\vartheta(H) \leq \tau =
2+2\rho^\ast+ \frac{1-2\rho^\ast}{1-\rho^\ast-(\rho^\ast)^2}.
\]
Combining this upper bound with the matching lower bound
\eqref{eq:H-lower} gives \eqref{eq: Lovasz number of H}
as asserted.
\end{proof}

\medskip
\subsubsection{\emph{Extension to every order larger than 10}}
\label{subsubsection: Extension to every order larger than ten}

\noindent 
Fix $n \geq 11$, set $t=n-10 \geq 1$, and define
\begin{equation}
\label{eq:join-construction}
 G_n \coloneqq G \vee K_t, \qquad H_n \coloneqq H \vee K_t,
\end{equation}
where $G$ and $H$ are shown in Figure~\ref{fig:G_and_H_regular_graphs}.
Both graphs are connected and have $n$ vertices. The ten original vertices 
of either $G$ or $H$ have degree $t+4=n-6$, whereas their $t$ added vertices have
degree $(t-1)+10=n-1$. Thus both graphs $G_n$ and $H_n$ are connected and irregular.

Since the adjacency matrices $\boldsymbol{A}(G)$ and $\boldsymbol{A}(H)$ are real 
symmetric and cospectral, and
\[
\boldsymbol{A}(G) \, \one_{10} = \boldsymbol{A}(H) \, \one_{10}=4 \, \one_{10}, 
\]
we may choose orthonormal eigenbases for them whose first vector is 
$\one_{10}/\sqrt{10}$ and whose eigenvalues occur in the same order. 
Mapping the eigenbasis for $H$ to that for $G$ gives an orthogonal 
matrix $\boldsymbol{U}$ satisfying
\begin{equation}\label{eq:seed-similarity}
 \boldsymbol{U}^{\mathsf T} \, \boldsymbol{A}(G) \, \boldsymbol{U} = \boldsymbol{A}(H), 
 \qquad \boldsymbol{U} \one_{10} = \one_{10}.
\end{equation}
The adjacency matrices of the graphs in \eqref{eq:join-construction} are given by 
\begin{equation}\label{eq:join-adjacencies}
 \boldsymbol{A}(G_n)=\begin{pmatrix}
 \boldsymbol{A}(G) & \boldsymbol{J}_{10\times t} \\ 
 \boldsymbol{J}_{t\times10} & \boldsymbol{J}_t-\boldsymbol{I}_t
 \end{pmatrix},\qquad
 \boldsymbol{A}(H_n)=\begin{pmatrix}
 \boldsymbol{A}(H) & \boldsymbol{J}_{10\times t} \\ 
 \boldsymbol{J}_{t \times 10} & \boldsymbol{J}_t-\boldsymbol{I}_t
 \end{pmatrix}.
\end{equation}
Consequently, $\boldsymbol{W} \coloneqq \diag(\boldsymbol{U}, \boldsymbol{I}_t)$ is 
orthogonal and satisfies
\begin{equation}\label{eq:join-similarity}
 \boldsymbol{W}^{\mathsf T} \boldsymbol{A}(G_n) \, \boldsymbol{W} = \boldsymbol{A}(H_n).
\end{equation}
The common degree matrix
\begin{equation}\label{eq:degree-matrix}
 \boldsymbol{D} = \diag \bigl( (t+4) \boldsymbol{I}_{10}, (t+9) \boldsymbol{I}_t \bigr)
\end{equation}
commutes with $\boldsymbol{W}$, as does $\boldsymbol{D}^{-1/2}$. Therefore the same orthogonal
similarity transforms each of
\[
 \boldsymbol{A}(G_n), \qquad \boldsymbol{D}-\boldsymbol{A}(G_n),
 \qquad \boldsymbol{D}+\boldsymbol{A}(G_n),\qquad
 \boldsymbol{I}_n - \boldsymbol{D}^{-1/2} \boldsymbol{A}(G_n) \, \boldsymbol{D}^{-1/2}
\]
into its counterpart for $H_n$. This proves simultaneous cospectrality for the 
adjacency, Laplacian, signless Laplacian, and normalized Laplacian matrices.

\smallskip 
The following lemma follows from the general formula for the Seidel 
characteristic polynomial of a disjoint union of regular graphs given 
in Theorem~1 and Section~4 of \cite{HaemersOboudi2020}.
For completeness, we include a direct proof based on an equitable 
partition. This lemma immediately establishes the Seidel cospectrality of 
\(G_n\) and \(H_n\) for each \(n \geq 11 \).

\begin{lemma}
\label{lemma: Seidel cospectral}
{\em Let \(G\) and \(H\) be regular graphs that are cospectral with respect to
their adjacency matrices. Then, for every integer \(r\geq 1\), the graphs
\(G \vee K_r\) and \(H \vee K_r\) are Seidel cospectral.}
\end{lemma}

\begin{proof}
Since \(G\) and \(H\) are regular and adjacency cospectral, they have the
same order \(n\), the same degree \(d\), and the same adjacency spectrum.
Write their common adjacency spectrum as
\[
d=\lambda_1 \geq \lambda_2 \geq \cdots \geq \lambda_n.
\]
For \(G\), choose an orthogonal adjacency eigenbasis
\(
\underline{x}_1,\ldots,\underline{x}_n
\)
such that \(\underline{x}_1=\underline{1}_n\) and
\[
\boldsymbol{A}(G)\underline{x}_i =\lambda_i \underline{x}_i.
\]
In particular,
\[
\underline{x}_i\perp\underline{1}_n,
\qquad 2\leq i\leq n.
\]
This choice is possible even if the largest adjacency eigenvalue \(d\) has 
multiplicity greater than one, so no connectedness assumption is required.

\smallskip
\noindent 
Recall that the Seidel matrix of a graph \(X\) of order \( m \) is given by 
\[
\boldsymbol{S}(X) = \boldsymbol{J}_m - \boldsymbol{I}_m -2 \boldsymbol{A}(X).
\]
With respect to the partition
\[
V(G\vee K_r) = V(G) \mathbin{\dot\cup} V(K_r),
\]
we have
\begin{equation}
\label{eq1: Seidel}
\boldsymbol{S}(G\vee K_r)=
\begin{pmatrix}
\boldsymbol{J}_n-\boldsymbol{I}_n-2\boldsymbol{A}(G)
    &-\boldsymbol{J}_{n\times r}\\
-\boldsymbol{J}_{r\times n}
    &\boldsymbol{I}_r-\boldsymbol{J}_r
\end{pmatrix}.
\end{equation}
For every \(i \in \{2,\ldots,n\}\), the orthogonality relation
\(\underline{x}_i\perp\underline{1}_n\) gives
\(\boldsymbol{J}_n \, \underline{x}_i = \underline{0}_n\). Therefore,
\[
\boldsymbol{S}(G\vee K_r)
\begin{pmatrix}
\underline{x}_i\\
\underline{0}_r
\end{pmatrix}
= (-1-2\lambda_i)
\begin{pmatrix}
\underline{x}_i\\
\underline{0}_r
\end{pmatrix}.
\]
Hence,
\[
-1-2\lambda_2, \ldots, -1-2\lambda_n
\]
are \(n-1\) Seidel eigenvalues of \(G\vee K_r\), counted according to
the chosen eigenbasis.

\noindent
Similarly, if \(\underline{y}\in\mathbb{R}^r\) satisfies
\(\underline{y}\perp\underline{1}_r\), then
\(\boldsymbol{J}_r \, \underline{y}=\underline{0}_r\), and hence
\[
\boldsymbol{S}(G\vee K_r)
\begin{pmatrix}
\underline{0}_n\\
\underline{y}
\end{pmatrix}
=
\begin{pmatrix}
\underline{0}_n\\
\underline{y}
\end{pmatrix}.
\]
Thus, there are \(r-1\) linearly independent eigenvectors corresponding
to the Seidel eigenvalue \(1\).

The preceding eigenvectors span two mutually orthogonal subspaces of
dimensions \(n-1\) and \(r-1\). They therefore account for \(n+r-2\)
eigenvalues of the \((n+r) \times (n+r)\) Seidel matrix. There are two 
further eigenvalues to determine, counted with multiplicity.

We use the following fact about equitable partitions.
Let $\boldsymbol{M}$ be a square matrix whose rows and columns
are indexed by the same set, partitioned into $k$ nonempty cells
$C_1,\ldots,C_k$. Suppose that every block of $\boldsymbol{M}$
has constant row sum, and let $\boldsymbol{Q}$ be the
$k\times k$ quotient matrix whose $(i,j)$-entry is the common
row sum of the $(i,j)$-block. Let $\boldsymbol{R}$ be the
characteristic matrix of the partition, whose $(x,i)$-entry
is $1$ if $x \in C_i$ and $0$ otherwise. Then
\begin{equation}
\label{eq: equitable partition}
\boldsymbol{M} \boldsymbol{R}
=\boldsymbol{R} \boldsymbol{Q}.
\end{equation}
Let \(\operatorname{col}(\boldsymbol{R})\) denote the subspace
spanned by the columns of \(\boldsymbol{R}\). It follows from 
\eqref{eq: equitable partition} that every eigenvalue
of \(\boldsymbol{Q}\) is an eigenvalue of \(\boldsymbol{M}\), with 
at least the same algebraic multiplicity. Moreover, if \(\underline{v}\) 
is an eigenvector of \(\boldsymbol{Q}\), then
\(\boldsymbol{R}\underline{v}\) is a corresponding eigenvector
of \(\boldsymbol{M}\) lying in \(\operatorname{col}(\boldsymbol{R})\).

We apply this fact to \(\boldsymbol{M} = \boldsymbol{S}(G\vee K_r)\), 
as given in \eqref{eq1: Seidel}.
Since \(G\) is \(d\)-regular, the four blocks of
\(\boldsymbol{S}(G\vee K_r)\) have respective constant row sums
\( n-1-2d, \, -r, \, -n \), and \( 1-r \).
Thus, the partition is equitable, with quotient matrix
\begin{equation}
\label{quotient matrix}
\boldsymbol{Q} =
\begin{pmatrix}
n-1-2d & -r\\
-n     & 1-r
\end{pmatrix}.
\end{equation}
Its characteristic matrix is
\begin{equation}
\label{eq: characteristic matrix}
\boldsymbol{R} =
\begin{pmatrix}
\underline{1}_n & \underline{0}_n\\
\underline{0}_r & \underline{1}_r
\end{pmatrix},
\end{equation}
and by \eqref{eq: equitable partition}
\begin{equation}
\label{eq2: equitable partition}
\boldsymbol{S}(G\vee K_r) \boldsymbol{R}
=\boldsymbol{R} \boldsymbol{Q}.
\end{equation}
The two previously constructed subspaces are
\begin{equation}
\label{eq: U, V subspaces}
\mathcal{U}
=\left\{
\begin{pmatrix} \underline{x}\\ \underline{0}_r \end{pmatrix}
:\underline{x}\perp\underline{1}_n
\right\},
\qquad
\mathcal{V}
=\left\{
\begin{pmatrix} \underline{0}_n\\ \underline{y} \end{pmatrix}
:\underline{y}\perp\underline{1}_r
\right\}.
\end{equation}
A vector is orthogonal to both subspaces if and only if its first
block is a scalar multiple of \(\underline{1}_n\) and its second
block is a scalar multiple of \(\underline{1}_r\). Hence, by 
\eqref{eq: characteristic matrix} and \eqref{eq: U, V subspaces}
\begin{equation}
\label{eq: span of col(R)}
(\mathcal{U}\oplus\mathcal{V})^\perp
= \operatorname{col}(\boldsymbol{R}).
\end{equation}
By \eqref{eq2: equitable partition}, \(\boldsymbol{Q}\) represents
the action of \(\boldsymbol{S}(G\vee K_r)\) on
\(\operatorname{col}(\boldsymbol{R})\), using the columns of
\(\boldsymbol{R}\) as a basis. Hence, the two remaining Seidel eigenvalues,
counted with multiplicity, are precisely the eigenvalues of
\(\boldsymbol{Q}\). They are the zeros of
\(
\det(x\boldsymbol{I}_2-\boldsymbol{Q})
= \bigl(x-(n-1-2d)\bigr)\bigl(x-(1-r)\bigr)-nr.
\)
It follows that the Seidel spectrum of \(G\vee K_r\) is determined
entirely by \(n\), \(d\), \(r\), and the adjacency spectrum of \(G\).
These data are the same for \(G\) and \(H\). Therefore,
\(G \vee K_r\) and \(H \vee K_r\) are Seidel cospectral.
\end{proof}

\begin{remark}
\label{remark: Seidel}
{\em Lemma~\ref{lemma: Seidel cospectral} does not appear to have been stated explicitly 
in this particular form in~\cite{HaemersOboudi2020}. Nevertheless, it follows directly 
from the results therein. Indeed,
\begin{equation}
\label{eq1: 4.10.26}
\overline{G \vee K_r}
= \overline{G} \sqcup \overline{K_r}
= \overline{G} \sqcup (rK_1),
\end{equation}
where $\sqcup$ denotes disjoint union of graphs.
Because \(G\) and \(H\) are regular and adjacency cospectral, their
complements \(\overline{G}\) and \(\overline{H}\) are also regular and
adjacency cospectral. Applying the Seidel characteristic-polynomial formula
of Haemers and Oboudi~\cite{HaemersOboudi2020} to \( \overline{G} \sqcup (rK_1) \) 
and \( \overline{H} \sqcup (rK_1) \) shows that these two disjoint unions are 
Seidel cospectral. The conclusion then follows from the sign-reversal identity 
for the Seidel matrix under graph complementation:
\begin{equation}
\label{eq: antisymmetry - Seidel}
\boldsymbol{S}(\overline{X})=-\boldsymbol{S}(X).
\end{equation}
Thus, Lemma~\ref{lemma: Seidel cospectral} is a consequence of
Theorem~1 and Section~4 of \cite{HaemersOboudi2020}.}
\end{remark}

An independent set in a join lies entirely in one of its two parts.
Maximum cliques combine across the parts, and any proper coloring uses
disjoint color sets on them. Therefore, by \eqref{eq:seed-alpha-omega},
\eqref{eq:seed-chi}, and \eqref{eq:join-construction}
\begin{equation}
\label{eq2: 4.10.26}
 \alpha(G_n)=\alpha(H_n)=3, \qquad \omega(G_n)=\omega(H_n)=t+3, 
 \qquad \chi(G_n)=\chi(H_n)=t+4.
\end{equation}
Substituting $t=n-10$ gives 
\begin{equation}
\label{eq3: 4.10.26}
 \omega(G_n)=\omega(H_n)=n-7, \qquad \chi(G_n)=\chi(H_n)=n-6.
\end{equation}

\noindent 
Complementation turns the joins in \eqref{eq:join-construction} into
disjoint unions:
\begin{equation}
\label{eq4: 4.10.26}
 \overline{G_n}=\overline G \sqcup (t K_1),\qquad
 \overline{H_n}=\overline H \sqcup (t K_1),
\end{equation}
where $t K_1 = \overline{K_t}$ 
consists of $t$ isolated vertices. Adding isolated vertices does not 
change the chromatic number of either seed complement. Thus
\eqref{eq:seed-chi-complement} and \eqref{eq4: 4.10.26} yield 
\begin{equation}
\label{eq5: 4.10.26}
 \chi(\overline{G_n})=\chi(\overline{H_n})=4.
\end{equation}

The following lemma serves as a preparatory step towards determining the 
maximum-cut numbers of the join graphs \( G_n \) and \( H_n \). 
It is a direct consequence of the decomposition of cuts under the graph join. 
It can also be obtained by specializing the join formula for the bivariate 
Ising polynomial given in Theorem~3.2 of \cite{AndrenMarkstrom2009}. 
Nevertheless, following Definition~\ref{definition: cardinality-constrained max-cut}, 
we provide a simple self-contained proof.
\begin{lemma}
\label{lemma: max-cut of F join Kr}
{\em Let \(F\) be a finite simple graph of order \(n\), and let \(r \geq 1\) be 
an integer. Then
\begin{align}
\label{eq: max-cut of F join Kr}
\mathrm{mc}(F \vee K_r) 
= \max_{k \in \{ 0, \ldots, n\}} 
\left\{ \mathrm{mc}_k(F) + kr + \max_{j \in \{0, \ldots, r \}} j(n+r-2k-j) \right\}.
\end{align}}
\end{lemma}

\begin{proof}
Consider a cut of \(F \vee K_r\), and let \(S\) be one of its two vertex
classes. Define 
\[
A \coloneqq S\cap V(F), \qquad B \coloneqq S\cap V(K_r),
\]
and write \( |A|=k \) and \( |B|=j \).
The edges of $F \vee K_r$ that cross the cut fall into three classes:
\begin{enumerate}
\item 
The edges of \(F\) joining \(A\) to \(V(F)\setminus A\).
Their number satisfies 
\[
e_F(A, V(F)\setminus A) \leq \mathrm{mc}_k(F).
\]
\item The edges of \(K_r\) joining \(B\) to \(V(K_r)\setminus B\).
Since \(K_r\) is complete, their number is \(j(r-j)\).
\item The edges between \(F\) and \(K_r\) joining \(A\) to
\(V(K_r) \setminus B\), or \(V(F) \setminus A\) to \(B\).
Their number is \( k(r-j)+(n-k)j \).
\end{enumerate}
Thus, every cut with \(|A|=k\) and \(|B|=j\) has size at most
\[
\mathrm{mc}_k(F)+j(r-j)+k(r-j)+(n-k)j.
\]
Conversely, for each \(k \in \{0, \ldots, n\} \) and
\(j \in \{0,\ldots,r\} \), choose a set \(A \subseteq V(F)\) of
cardinality \(k\) attaining \(\mathrm{mc}_k(F)\), and any set
\(B \subseteq V(K_r)\) of cardinality \(j\).
The cut with vertex class \( S \coloneqq A \, \dot\cup \, B\) has exactly
\begin{align*}
&\mathrm{mc}_k(F)+j(r-j)+k(r-j)+(n-k)j\\
&= \mathrm{mc}_k(F)+ kr + j(n+r-2k-j)
\end{align*}
crossing edges. Hence, the bound is attained for every pair \((k,j)\).
Maximizing over \(k\) and \(j\) proves \eqref{eq: max-cut of F join Kr}.
\end{proof}

\begin{proposition}
\label{proposition: max-cuts for n at least 11}
{\em Let \(G\) and \(H\) be the two graphs of order \(10\) in Figure~\ref{fig:G_and_H_regular_graphs}.
For every \(n\geq 11\), let
\begin{align}
\label{eq1: 17.09.26}
G_n=G\vee K_{n-10},  \qquad  H_n=H\vee K_{n-10},
\end{align}
and
\begin{align}
\label{eq2: 17.09.26}
\widetilde{G}_n=\overline{G}\vee K_{n-10},
\qquad
\widetilde{H}_n=\overline{H}\vee K_{n-10}.
\end{align}
Then
\begin{align}
\label{eq3: 17.09.26}
\mathrm{mc}(G_n)=\mathrm{mc}(H_n) =
\begin{cases}
\displaystyle \left\lfloor\frac{n^2}{4}\right\rfloor-9,
   & 11\leq n\leq 14,\\[6mm]
10(n-10),
   & 15\leq n\leq 18,\\[3mm]
\displaystyle \left\lfloor\frac{n^2}{4}\right\rfloor,
   & n\geq 19,
\end{cases}
\end{align}
and
\begin{align}
\label{eq4: 17.09.26}
\mathrm{mc}(\widetilde{G}_n) =
\mathrm{mc}(\widetilde{H}_n) =
\begin{cases}
\displaystyle \left\lfloor\frac{n^2}{4}\right\rfloor-6,
   & 11\leq n\leq 15,\\[6mm]
10(n-10),
   & 16\leq n\leq 18,\\[3mm]
\displaystyle \left\lfloor\frac{n^2}{4}\right\rfloor,
   & n\geq 19.
\end{cases}
\end{align}
Consequently, for every \(n\geq 19\),
\begin{align}
\label{eq5: 17.09.26}
\mathrm{mc}(G_n)
= \mathrm{mc}(H_n)
= \mathrm{mc}(\widetilde{G}_n)
= \mathrm{mc}(\widetilde{H}_n)
= \mathrm{mc}(K_n).
\end{align}}
\end{proposition}

\begin{proof} 
Set \( r=n-10 \). Since \(n\geq 11\), we have \(r\geq 1\). By
Lemma~\ref{lemma: max-cut of F join Kr}, for every graph \(X\) 
of order~10,
\begin{align}
\label{eq9: 17.09.26}
\mathrm{mc}(X \vee K_r)
= \max_{0\leq k\leq 10}
\left\{ \mathrm{mc}_k(X)+kr+ \max_{0\leq j\leq r} j(10+r-2k-j) \right\}.
\end{align}
Since, by symmetry, 
\(
\mathrm{mc}_k(X)=\mathrm{mc}_{10-k}(X),
\)
we may restrict the maximization to \(0\leq k\leq 5\). 
The quadratic function
\[
j\longmapsto j(r+10-2k-j)
\]
is concave. Its maximum over the integers \(0\leq j\leq r\) is attained
at the integer or integers closest to \( \tfrac12 (r+10-2k) \), 
unless this number is larger than \(r\), in which case the maximum is
attained at \(j=r\). Thus, for \(0\leq k\leq 5\), define
\begin{align}
\label{eq10: 17.09.26}
F_k^{X}(r)=
\begin{cases}
\mathrm{mc}_k(X)+r(10-k), & r\leq 10-2k,\\[2mm]
\displaystyle
\mathrm{mc}_k(X)+kr+\left\lfloor\frac{(r+10-2k)^2}{4}\right\rfloor, & r\geq 10-2k.
\end{cases}
\end{align}
Then
\begin{align}
\label{eq11: 17.09.26}
\mathrm{mc}(X\vee K_r)=\max_{0\leq k\leq 5} F_k^{X}(r).
\end{align}
We first apply \eqref{eq11: 17.09.26} to \(X\in\{G,H\}\). By
Proposition~\ref{proposition: max-cut of regular G, H and complements},
\begin{align}
\label{eq12: 17.09.26}
\bigl(\mathrm{mc}_k(X)\bigr)_{k=0}^{5} = (0,4,8,12,14,16).
\end{align}
In particular, \(G\) and \(H\) have the same cardinality-constrained
cut profile, and hence
\begin{align}
\label{eq13: 17.09.26}
\mathrm{mc}(G_n)=\mathrm{mc}(H_n).
\end{align}
Substituting \eqref{eq12: 17.09.26} into \eqref{eq11: 17.09.26} and comparing
\(F_0^{X}(r),\ldots,F_5^{X}(r)\) gives
\[
\max_{0\leq k\leq 5} F_k^{X}(r) =
\begin{cases}
\displaystyle
\left\lfloor\frac{(r+10)^2}{4}\right\rfloor-9, & 1\leq r\leq 4, \\[5mm]
10r, & 5\leq r\leq 8,\\[2mm]
\displaystyle
\left\lfloor\frac{(r+10)^2}{4}\right\rfloor, & r\geq 9,
\end{cases}
\]
where the first line holds by the equality 
\[
16+5r+\left\lfloor\frac{r^2}{4}\right\rfloor
= \left\lfloor\frac{(r+10)^2}{4}\right\rfloor-9.
\]
Since \(n=r+10\), this proves \eqref{eq3: 17.09.26}.

Similarly, \eqref{eq4: 17.09.26} is proved by applying the same argument to
\(X\in\{\overline{G},\overline{H}\}\) and relying on 
Proposition~\ref{proposition: max-cut of regular G, H and complements}.

Finally, a cut of \(K_n\) whose two vertex classes have cardinalities
\(s\) and \(n-s\) has \(s(n-s)\) crossing edges, where \( s \in \{0, 1, \ldots, n\} \). 
Therefore,
\[
\mathrm{mc}(K_n) = \max_{s \in \{0, 1, \ldots, n\}} s(n-s) 
= \left\lfloor\frac{n^2}{4}\right\rfloor,
\]
which, for \(n \geq 19\), yields the last equality in \eqref{eq5: 17.09.26}.
\end{proof}

The following lemma specializes a result from Section~18 of \cite{Knuth94} for our purposes. 
\begin{lemma}
\label{lemma: Lovasz theta for F join Kt}
{\em For every finite simple graph \(F\) and every
integer \(t \geq 1\),
\begin{equation}
\label{eq:theta-join}
    \vartheta(F \vee K_t) = \vartheta(F).
\end{equation}}
\end{lemma}
\begin{proof}
The Lov\'{a}sz theta function satisfies (see Section~18 of \cite{Knuth94})
\begin{align}
\label{eq: Knuth - join of graphs}
\vartheta(F_1 \vee F_2)=\max\{\vartheta(F_1), \vartheta(F_2)\},
\end{align}
for arbitrary finite simple graphs \( F_1 \) and \( F_2 \). Set \( F_1 \coloneqq F \)
and \( F_2 \coloneqq K_t \). Since \(\vartheta(K_t)=1\) and
\(\vartheta(F) \geq \alpha(F) \geq 1\), equality \eqref{eq:theta-join} follows
from \eqref{eq: Knuth - join of graphs}. 
\end{proof}

Applying \eqref{eq:theta-join} to \eqref{eq:join-construction} and combining 
the result with Propositions~\ref{proposition: Lovasz number of G} 
and~\ref{proposition: Lovasz number of H} yields 
\begin{align} \label{eq:separation-final}
\begin{aligned}
\vartheta(G_n) = \vartheta(G) = 3.23606797749979\ldots,  \quad 
\vartheta(H_n) = \vartheta(H) = 3.26880089469821\ldots.
\end{aligned}
\end{align}
Hence, for every integer $n\geq 11$, the graphs $G_n$ and $H_n$
are connected, irregular, nonisomorphic graphs on $n$ vertices
that are cospectral with respect to the adjacency, Laplacian,
signless Laplacian, normalized Laplacian, and Seidel matrices,
while having equal independence, clique, chromatic, complement
chromatic, and maximum-cut numbers but distinct
Lov\'{a}sz $\vartheta$-numbers.  
\end{proof}

\subsection{Further Existence and Nonexistence Results}
\label{subsection: Further existence and nonexistence results}

\begin{theorem} \label{theorem: graphs of order at most 9}
{\em There is no pair of connected, irregular, nonisomorphic graphs on at
most nine vertices that are simultaneously cospectral with respect to
the adjacency, Laplacian, signless Laplacian, and normalized Laplacian
matrices, and have identical independence, clique, chromatic, and
complement chromatic numbers.}
\end{theorem}

\begin{proof}
We prove the assertion by an exhaustive computation in SageMath~9.3.
We describe the enumeration, the quantities computed for each graph,
and the test that excludes a pair with the stated properties.

\medskip
\noindent\textbf{Enumeration of the graphs.}
For each fixed order \(m\), the SageMath iterator
\texttt{graphs(m)} generates one representative from every
isomorphism class of finite simple undirected graphs on \(m\)
vertices. Thus, the instruction \texttt{for X in graphs(m)}
is the enumeration step: it successively supplies these
representatives to the body of the loop. It does not perform random
sampling, and it need not print or store the complete list of
generated graphs. The iterator uses isomorph-free exhaustive
generation by canonical augmentation, which avoids generating
multiple representatives of the same isomorphism class.

The outer loop \texttt{range(1, 10)} runs through precisely the
orders \(m=1,\ldots,9\), since the upper endpoint is excluded.
For each generated graph \(X\), the predicate
\texttt{X.is\_connected()} tests whether \(X\) is connected, and
\texttt{X.is\_regular()} tests whether all its vertices have the same
degree. A disconnected or regular graph is discarded by the
instruction \texttt{continue}, which advances to the next graph.
These tests are applied after generation, rather than used to prune
the generation process. Consequently, every connected, irregular
isomorphism class of order \(m\) is retained exactly once.

The connected graphs of orders one and two are regular and are
therefore discarded. Every retained graph has at least three vertices
and, by connectivity, has no isolated vertex.

\medskip
\noindent\textbf{Exact computation of the spectral data.}
For a retained graph \(X\) on \(m\) vertices, let
\[
\boldsymbol{A}(X), \qquad
\boldsymbol{L}(X)=\boldsymbol{D}(X)-\boldsymbol{A}(X), \qquad
\boldsymbol{Q}(X)=\boldsymbol{D}(X)+\boldsymbol{A}(X)
\]
denote its adjacency, Laplacian, and signless Laplacian matrices,
respectively. Here, \(\boldsymbol{D}(X)\) is the diagonal matrix of
vertex degrees, with the same vertex ordering as
\(\boldsymbol{A}(X)\). Its normalized Laplacian is
\[
\boldsymbol{\mathcal{L}}(X)
= \boldsymbol{I}_m - \boldsymbol{D}(X)^{-1/2} \boldsymbol{A}(X) \boldsymbol{D}(X)^{-1/2}.
\]
Since every retained graph has positive vertex degrees,
\(\boldsymbol{D}(X)\) is invertible, and
\[
\boldsymbol{D}(X)^{-1/2} \boldsymbol{\mathcal{L}}(X) \boldsymbol{D}(X)^{1/2}
= \boldsymbol{I}_m - \boldsymbol{D}(X)^{-1}\boldsymbol{A}(X).
\]
Thus, \(\boldsymbol{\mathcal{L}}(X)\) is similar to the rational
matrix on the right-hand side. Its characteristic polynomial can
therefore be computed exactly over \(\mathbb{Q}\), without
introducing square roots or numerical approximations.

In the code, \texttt{QQ} denotes the rational field. The command  
\begin{center}
\texttt{X.adjacency\_matrix().change\_ring(QQ)}
\end{center}
constructs the adjacency matrix over this field, while
\texttt{diagonal\_matrix(QQ, X.degree())} constructs the degree
matrix. For the integer-labeled graphs produced by
\texttt{graphs(m)}, the default orderings used by these two commands
agree. The matrices \texttt{L}, \texttt{Q}, and \texttt{NL} then
represent, respectively, the Laplacian, the signless Laplacian, and
the rational matrix similar to the normalized Laplacian.

For a square matrix \(\mathbf{M}\), let
\(
\phi_{\mathbf{M}}(x)\coloneqq\det(xI-\mathbf{M})
\)
denote the characteristic polynomial of \(\mathbf{M}\).
The command \texttt{M.charpoly()} computes this polynomial.
The subsequent command \texttt{.list()} returns its coefficients
in ascending powers of \(x\), including internal zero
coefficients. Converting this list to a tuple produces an immutable
representation suitable for use as part of a dictionary key.
For example, \(x^3-2x+1\) is represented by
\texttt{(1, -2, 0, 1)}. Since characteristic polynomials are monic,
their degrees and all their coefficients are recorded by these
tuples. All four characteristic polynomials are computed and
compared exactly over the field \( \mathbb{Q}\) of rational numbers.

\medskip
\noindent\textbf{Computation of the additional graph invariants.}
The command \texttt{X.independent\_set()} returns a largest 
independent set, whose cardinality is the independence number \(\alpha(X)\).
The command \texttt{X.clique\_number()} returns \(\omega(X)\), and
the command \texttt{X.chromatic\_number()} returns the chromatic
number \(\chi(X)\).
Finally, \texttt{X.complement()} constructs the complement
graph of $X$, and applying the chromatic-number routine
to this graph gives \(\chi(\overline{X})\).
These routines are used to obtain the stated integer invariants,
not heuristic estimates or bounds.

For every retained graph \(X\), we compute the augmented signature
\[
\Sigma(X)\coloneqq\bigl(
\phi_{\boldsymbol{A}(X)}(x), \, 
\phi_{\boldsymbol{L}(X)}(x), \, 
\phi_{\boldsymbol{Q}(X)}(x), \, 
\phi_{\boldsymbol{\mathcal{L}}(X)}(x), \, 
\alpha(X), \, \omega(X), \, \chi(X), \, \chi(\overline{X}) \bigr).
\]
The function \texttt{spectral\_signature} below returns all eight
components of this signature. Equality of the first four
components is equivalent to equality of the corresponding
eigenvalue multisets, including multiplicities. Equality of the
remaining components is precisely equality of the four additional
invariants in the theorem. Moreover, all eight components are
invariant under graph isomorphism.

\medskip
\noindent\textbf{Detection of repeated signatures.}
For each order \(m\), the dictionary \texttt{signatures} is
initialized to be empty. It associates each signature encountered
with the first graph producing that signature.
When the signature of a subsequently generated graph is already
present, the two graphs are recorded in
\texttt{repeated\_signatures}. Otherwise, the new signature is
inserted into the dictionary.

The command \texttt{X.canonical\_label()} supplies a canonically
labeled copy of \(X\) for storage. It is not the enumeration
procedure: the fact that the graphs being compared represent
distinct isomorphism classes follows from the iterator
\texttt{graphs(m)}. Dictionary membership tests compare the complete
tuple keys, so an internal hash collision between unequal keys
cannot create a false equality of signatures.

After each iteration, the dictionary contains exactly the distinct
signatures of the retained graphs processed so far. This follows
inductively from the insertion and membership tests. Hence, a
repeated signature is detected if and only if two of the retained
isomorphism-class representatives have the same augmented
signature.

\medskip
\noindent\textbf{Implementation and output.}
The following SageMath~9.3 code implements these steps.

\begin{verbatim}
def spectral_signature(X):
    # Adjacency and degree matrices over the rational field.
    m = X.order()
    A = X.adjacency_matrix().change_ring(QQ)
    D = diagonal_matrix(QQ, X.degree())
    L = D - A
    Q = D + A

    # All retained graphs have positive vertex degrees.
    # NL is rational and similar to the normalized Laplacian.
    NL = identity_matrix(QQ, m) - D.inverse() * A

    # Four characteristic polynomials, encoded by coefficient
    # tuples, followed by the four combinatorial invariants.
    return (
        tuple(A.charpoly().list()),
        tuple(L.charpoly().list()),
        tuple(Q.charpoly().list()),
        tuple(NL.charpoly().list()),
        len(X.independent_set()),  # alpha(X)
        X.clique_number(),         # omega(X)
        X.chromatic_number(),
        X.complement().chromatic_number()
    )

# Examine each order from 1 through 9.
for m in range(1, 10):
    signatures = {}
    repeated_signatures = []

    # One representative of every isomorphism class of order m.
    for X in graphs(m):
        if not X.is_connected():
            continue
        if X.is_regular():
            continue

        signature = spectral_signature(X)

        if signature in signatures:
            # A different isomorphism class has the same signature.
            repeated_signatures.append(
                (signatures[signature], X.canonical_label())
            )
        else:
            # Store the first graph with this signature.
            signatures[signature] = X.canonical_label()

    # Print only after the enumeration for this order is complete.
    print(m, len(repeated_signatures))
\end{verbatim}

\noindent The output of the completed computation consists of the 
following pairs: 
\[
1 \; 0 
\quad 2 \; 0
\quad 3 \; 0
\quad 4 \; 0
\quad 5 \; 0
\quad 6 \; 0
\quad 7 \; 0
\quad 8 \; 0
\quad 9 \; 0.
\]
The first entry in each pair is the graph order \(m\), and the
second is the number of repeated-signature occurrences.
More precisely, if a signature occurs for \(r\) retained graphs,
it contributes \(r-1\) entries to
\texttt{repeated\_signatures}: every occurrence after the first is
paired with the first stored representative. Thus, the program
does not necessarily list all \(\binom{r}{2}\) pairs, but its
reported count is zero if and only if no signature occurs more
than once.

It follows that \(\Sigma\) is injective on the connected, irregular
isomorphism classes of every order \(m\leq9\). Indeed, if two
nonisomorphic graphs of the same order satisfied all the
conditions in the theorem, their representatives would both be
generated, would pass the two filters, and would have equal
signatures. The later representative would then produce a
repeated-signature occurrence, contradicting the reported zero.
Graphs of different orders cannot be adjacency-cospectral,
because their characteristic polynomials have different degrees.
Therefore the separate computations for \(m=1,\ldots,9\) cover
all possible pairs and establish the assertion.
\end{proof}

Combining the nonexistence result of Theorem~\ref{theorem: graphs of order at most 9} 
with the construction for order \(10\) given below and the all-orders construction of
Theorem~\ref{theorem: all-orders} for \(n\geq 11\) yields the following sharp result.
\begin{theorem} 
\label{theorem: minimum-order}
{\em Let \(n_{\min}\) denote the smallest order of a pair of connected,
irregular, nonisomorphic graphs that are cospectral with respect to the
adjacency, Laplacian, signless Laplacian, normalized Laplacian, and Seidel
matrices, have the same independence number, clique number, chromatic
number, maximum-cut number, and complement chromatic number, 
but have distinct Lov\'{a}sz \(\vartheta\)-numbers. Then \( n_{\min}=10\), 
and such a pair exists for every integer \(n \geq n_{\min}\).}
\end{theorem}

\begin{proof}
Theorem~\ref{theorem: graphs of order at most 9} asserts that no such pair 
exists on fewer than ten vertices, and hence \[ n_{\min}\geq 10. \]
On the other hand, a computer search produces a pair of connected,
irregular graphs on ten vertices satisfying all the stated
cospectrality and combinatorial-invariant conditions and having
distinct Lov\'asz \(\vartheta\)-numbers. One such pair, obtained by a similar 
SageMath computation \cite{SageMath}, is presented in 
Figure~\ref{fig:G_and_H_irregular_graphs}.

\begin{figure}[htbp]
    \centering
    \includegraphics[width=1.00\linewidth]%
        {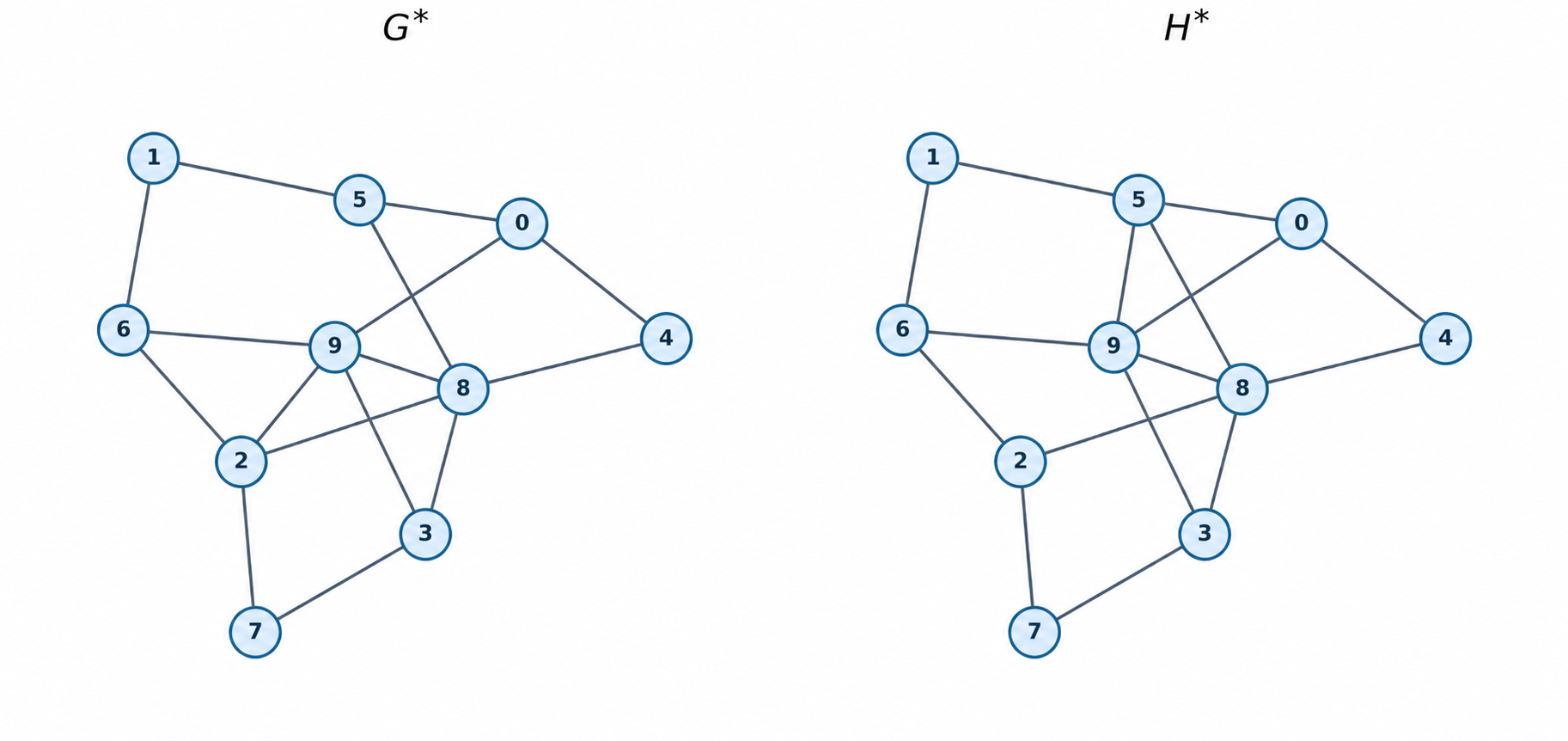}
    \caption{A pair of connected, irregular, nonisomorphic graphs on
    ten vertices satisfying the cospectrality and
    combinatorial-invariant conditions of
    Theorem~\ref{theorem: minimum-order}, while having distinct
    Lov\'{a}sz \(\vartheta\)-numbers.}
    \label{fig:G_and_H_irregular_graphs}
\end{figure}

The graphs \(G^\ast\) and \(H^\ast\) shown in Figure~\ref{fig:G_and_H_irregular_graphs} 
are connected, nonisomorphic, and irregular graphs on~10 vertices, and they 
have the common degree sequence
\[
    (5,5,4,3,3,3,3,2,2,2).
\]
Note that \(H^\ast\) is obtained from \(G^\ast\) by deleting the edge \(\{2,9\}\)
and adding the edge \(\{5,9\}\). The nonisomorphism of the graphs \(G^\ast\) and 
\(H^\ast\) is justified as follows.  
\begin{itemize}
\item Each graph has a unique vertex of degree \(4\).
\item In \(G^\ast\), this vertex is \(2\). Its unique neighbor of degree \(3\) is
vertex \(6\), whose unique neighbor of degree \(2\) is vertex \(1\).
The other neighbor of vertex \(1\) has degree \(3\).

\item In \(H^\ast\), the unique vertex of degree \(4\) is vertex \(5\). Its unique
neighbor of degree \(3\) is vertex~\(0\), whose unique neighbor of
degree \(2\) is vertex \(4\). However, the other neighbor of vertex
\(4\) has degree \(5\).
\end{itemize}
Since an isomorphism preserves adjacency and vertex degrees, this
degree-defined local property must also be preserved. It differs
between the two graphs, and therefore \( G^\ast \not\cong H^\ast \).

\noindent 
The common adjacency characteristic polynomial of $G^\ast$ and $H^\ast$ is
\[
P_{\boldsymbol{A}}(x)
 =x^{10}-16x^8-6x^7+65x^6+14x^5-86x^4+6x^3+14x^2,
\]
their common Laplacian characteristic polynomial is
\[
\begin{aligned}
P_{\boldsymbol{L}}(x)
 ={}&x^{10}-32x^9+439x^8-3382x^7+16089x^6-48888x^5 \\
   &+94580x^4-111898x^3+73152x^2-20080x,
\end{aligned}
\]
their common signless-Laplacian characteristic polynomial is
\[
\begin{aligned}
P_{\boldsymbol{Q}}(x)
 ={}&x^{10}-32x^9+439x^8-3394x^7+16317x^6-50696x^5 \\
   &+102344x^4-131402x^3+101776x^2-42704x+7360,
\end{aligned}
\]
their common normalized-Laplacian characteristic polynomial is
\[
\begin{aligned}
P_{\boldsymbol{\mathcal{L}}}(x)
 ={}&x^{10}-10x^9+\tfrac{78193}{1800}x^8
     -\tfrac{967}{9}x^7+\tfrac{899119}{5400}x^6
     -\tfrac{18077}{108}x^5 \\
   &+\tfrac{441239}{4050}x^4-\tfrac{29824}{675}x^3
     +\tfrac{2729}{270}x^2-\tfrac{2008}{2025}x,
\end{aligned}
\]
and their common Seidel characteristic polynomial is
\[
\begin{aligned}
P_{\boldsymbol{S}}(x) 
 ={}&x^{10} - 45x^8 - 8x^7 + 706x^6 + 312x^5 - 4514x^4\\
     &- 3672x^3 + 9389x^2 + 12584x + 3679.
\end{aligned}
\]     
Moreover, both graphs have independence number~\(4\), clique number~\(3\), 
chromatic number~\(3\), and complement chromatic number~\(5\). We next 
determine the cardinality-constrained maximum-cut profiles of the two graphs 
and their complements, and show that the profiles coincide within each pair.

\begin{proposition}
\label{proposition: irregular graphs - 10 vertices}
{\em Let \(G^\ast\) and \(H^\ast\) be the irregular, cospectral, 
and nonisomorphic graphs of order~\(10\) defined in 
Figure~\ref{fig:G_and_H_irregular_graphs}.
For every integer \(0\leq k\leq10\),
\[
\mathrm{mc}_k(G^\ast) = \mathrm{mc}_k(H^\ast),
\qquad
\mathrm{mc}_k(\overline{G^\ast}) = \mathrm{mc}_k(\overline{H^\ast}).
\]
Their cardinality-constrained maximum-cut profiles are
\[
\begin{array}{c|ccccccccccc}
k
& 0 & 1 & 2 & 3 & 4 & 5 & 6 & 7 & 8 & 9 & 10\\
\hline
\mathrm{mc}_k(G^\ast)
& 0 & 5 & 8 & 11 & 13 & 13 & 13 & 11 & 8 & 5 & 0\\
\mathrm{mc}_k(H^\ast)
& 0 & 5 & 8 & 11 & 13 & 13 & 13 & 11 & 8 & 5 & 0\\
\mathrm{mc}_k(\overline{G^\ast})
& 0 & 7 & 13 & 17 & 20 & 20 & 20 & 17 & 13 & 7 & 0 \\
\mathrm{mc}_k(\overline{H^\ast})
& 0 & 7 & 13 & 17 & 20 & 20 & 20 & 17 & 13 & 7 & 0 
\end{array}
\]
In particular,
\[
\mathrm{mc}(G^\ast) = \mathrm{mc}(H^\ast) = 13,
\qquad
\mathrm{mc}(\overline{G^\ast}) = \mathrm{mc}(\overline{H^\ast}) = 20.
\]}
\end{proposition}

\begin{proof}
For a graph \(X\) of order~\(10\) and a set \(S\subseteq V(X)\), let
\[
\delta_X(S)=e_X(S,V(X)\setminus S).
\]
By definition,
\[
\mathrm{mc}_k(X) = \max_{\substack{S\subseteq V(X)\\ |S|=k}} \delta_X(S).
\]
It is enough to consider \(0\leq k\leq5\), since a set and its
complement determine the same cut. Thus,
\begin{equation}
\label{eq: sym}
\mathrm{mc}_k(X)=\mathrm{mc}_{10-k}(X).
\end{equation}

For \(X\in\{G^\ast, H^\ast\}\), enumeration of all
\( \bigl(\genfrac{}{}{0pt}{}{\scriptstyle 10}{\scriptstyle k}\bigr) \) 
subsets of cardinality \(k\), for \(0\leq k\leq5\), gives the following 
maxima, together with subsets attaining them:
\[
\begin{array}{c|c|c|c}
k&
\max_{|S|=k}\delta_X(S)&
\text{A maximizing set in }G^\ast&
\text{A maximizing set in }H^\ast\\
\hline
0&0&\varnothing&\varnothing\\
1&5&\{8\}&\{8\}\\
2&8&\{0,8\}&\{0,8\}\\
3&11&\{0,6,8\}&\{0,6,8\}\\
4&13&\{0,6,7,8\}&\{0,6,7,8\}\\
5&13&\{0,1,2,3,8\}&\{0,1,6,7,8\}
\end{array}
\]
Hence,
\[
\bigl(\mathrm{mc}_k(G^\ast)\bigr)_{k=0}^{5}
= \bigl(\mathrm{mc}_k(H^\ast)\bigr)_{k=0}^{5}
= (0,5,8,11,13,13).
\]

To determine the cardinality-constrained maximum-cut profiles of the complementary graphs 
$\overline{G^\ast}$ and $\overline{H^\ast}$, observe that, for every \(S\subseteq V(X)\) with \(|S|=k\),
\[
\delta_{\overline X}(S)=k(10-k)-\delta_X(S).
\]
Consequently,
\[
\mathrm{mc}_k(\overline X)
= k(10-k) - \min_{\substack{S\subseteq V(X)\\ |S|=k}} \delta_X(S).
\]
The same enumeration gives the following minima, together with
subsets attaining them:
\[
\begin{array}{c|c|c|c}
k&
\min_{|S|=k}\delta_X(S)&
\text{A minimizing set in }G^\ast&
\text{A minimizing set in }H^\ast\\
\hline
0&0&\varnothing&\varnothing\\
1&2&\{1\}&\{1\}\\
2&3&\{0,4\}&\{0,4\}\\
3&4&\{0,1,5\}&\{1,2,6\}\\
4&4&\{0,1,4,5\}&\{1,2,6,7\}\\
5&5&\{0,1,4,5,6\}&\{0,1,5,6,9\}
\end{array}
\]
It follows that
\[
\bigl(\mathrm{mc}_k(\overline{G^\ast})\bigr)_{k=0}^{5}
= \bigl(\mathrm{mc}_k(\overline{H^\ast})\bigr)_{k=0}^{5}
= (0,7,13,17,20,20).
\]

Applying \eqref{eq: sym} to each of the four graphs gives the
displayed cardinality-constrained maximum-cut profiles for all \(0\leq k\leq10\).
Finally, taking the maximum over \(k\) yields
\[
\mathrm{mc}(G^\ast) = \mathrm{mc}(H^\ast) = 13,
\qquad
\mathrm{mc}(\overline{G^\ast}) = \mathrm{mc}(\overline{H^\ast}) =20.
\]
\end{proof}

The strict separation of the Lov\'{a}sz numbers of
\(G^\ast\) and \(H^\ast\) is established analytically in
Appendix~\ref{app:irregular-theta-separation}
(see Proposition~\ref{prop:irregular-theta-separation}
and Remark~\ref{rem:irregular-theta-separation}).
A numerical computation based on \eqref{eq:theta-sdp} gives the
following approximate values:
\begin{equation}
\label{eq1: 19.10.26}
\vartheta(G^\ast) \approx 4.23606798,
\qquad
\vartheta(H^\ast) \approx 4.30169434.
\end{equation}
The analytically established strict separation also certifies
that \(G^\ast\) and \(H^\ast\) are nonisomorphic.
Thus, the pair satisfies all the required conditions on ten
vertices, which gives \(n_{\min}\leq10\).
Combining the two bounds gives \(n_{\min}=10\).
Finally, the explicitly constructed pair
\(G^\ast,H^\ast\), whose properties were verified above,
provides the required pair for \(n=10\), while
Theorem~\ref{theorem: all-orders} provides such a pair for every
integer \(n\geq11\). Together, these results show that such a pair
exists for every integer \(n\geq n_{\min}\).
\end{proof}

\section{Related Discussion}
\label{section: discussion}

In this section, we discuss further properties of the two seed graphs
used in the proof of Theorem~\ref{theorem: all-orders}
(see Figure~\ref{fig:G_and_H_regular_graphs}). 

We begin with the chromatic polynomials. For a finite simple graph \(F\),
the chromatic polynomial, denoted by \(\chi_F(q)\), is the polynomial
whose value at each positive integer \(q\) counts the proper vertex
colorings of \(F\) with \(q\) labeled colors. Here, a coloring is proper
if adjacent vertices receive distinct colors.

The chromatic polynomials of \(G\) and \(H\) are given, respectively, by
\begin{equation}
\label{eq: chromatic polynomials G, H}
\begin{aligned}
\chi_G(q) &= q(q-1)(q-2)(q-3)
 \bigl(q^6-14q^5+87q^4-310q^3+676q^2-860q+496\bigr),\\
\chi_H(q) &= q(q-1)(q-2)(q-3)
 \bigl(q^6-14q^5+87q^4-310q^3+676q^2-860q+494\bigr).
\end{aligned}
\end{equation}
Thus, \(G\) and \(H\) have distinct chromatic polynomials, although
their chromatic numbers coincide. Indeed, both polynomials vanish at
\(q=1,2,3\), whereas $\chi_G(4)=1536>0$ and $\chi_H(4)=1488>0$.

The corresponding chromatic polynomials of the complements of $G$ and $H$ are 
{\small
\begin{equation}
\label{eq: chromatic polynomials G, H complements}
\begin{aligned}
\chi_{\overline{G}}(q) &= q(q-1)(q-2)(q-3)
 \bigl(q^6-19q^5+163q^4-801q^3+2357q^2-3903q+2818\bigr),\\
\chi_{\overline{H}}(q) &= q(q-1)(q-2)(q-3)
 \bigl(q^6-19q^5+163q^4-801q^3+2357q^2-3903q+2816\bigr).
\end{aligned}
\end{equation}
}
Both polynomials vanish at \(q=1,2,3\), while
$\chi_{\overline{G}}(4)=528>0$ and $\chi_{\overline{H}}(4)=480>0$.
Therefore, both complement graphs have chromatic number~4, in agreement with
\eqref{eq:seed-chi-complement}, but admit different numbers of proper colorings 
with four labeled colors.

We next consider the symmetry properties of the seed graphs 
in Figure~\ref{fig:G_and_H_regular_graphs}.
Neither of the $4$-regular, cospectral graphs $G$ and $H$ is
edge-transitive. Indeed, for any graph $X$ and any automorphism
$\varphi\in\operatorname{Aut}(X)$,
\[
\varphi\bigl(N_X(u)\cap N_X(v)\bigr)
= N_X(\varphi(u))\cap N_X(\varphi(v)).
\]
Consequently, the number of triangles containing an edge
$\{u,v\}$, namely $|N_X(u)\cap N_X(v)|$, is constant on each
orbit of edges under $\operatorname{Aut}(X)$.

In $G$, the edge $\{a,d\}$ is contained in three triangles, since
\begin{equation}
\label{eq8: 4.10.26}
N_G(a)\cap N_G(d)=\{b,f,h\},
\end{equation}
whereas the edge $\{a,b\}$ is contained in only one triangle, since
\begin{equation}
\label{eq9: 4.10.26}
N_G(a)\cap N_G(b)=\{d\}.
\end{equation}
Hence, no automorphism of $G$ maps $\{a,d\}$ to $\{a,b\}$,
so $G$ is not edge-transitive. Similarly, in $H$,
\[
N_H(a)\cap N_H(g)=\varnothing,
\qquad
N_H(a)\cap N_H(b)=\{d\}.
\]
Thus, no automorphism of $H$ maps $\{a,g\}$ to $\{a,b\}$, so $H$ is 
not edge-transitive. 

Neither graph is vertex-transitive either. Indeed, the number
of triangles containing a vertex is invariant under graph
automorphisms. In $G$, vertex $a$ belongs to three triangles,
whereas vertex $b$ belongs to two. In $H$, vertex $a$ belongs
to two triangles, whereas vertex $d$ belongs to three.
Thus, in each graph, these vertices lie in distinct orbits
under the automorphism group, ruling out vertex-transitivity.

Turning to spectral bounds, the ratio bound (see Theorem~9 of
\cite{Lovasz1979}) is not tight for the Lov\'asz theta function of
either graph. Since \(G\) and \(H\) are cospectral, their least
adjacency eigenvalues satisfy
\[
\lambda_{\min}(G)=\lambda_{\min}(H)
=-\frac{1+\sqrt{17}}{2}.
\]
Both graphs are \(4\)-regular on \(10\) vertices, so the ratio bound
gives
\[
\vartheta(G),\ \vartheta(H)
\leq \frac{-10\lambda_{\min}(G)}{4-\lambda_{\min}(G)}
=\frac{5(\sqrt{17}-1)}{4}
\approx 3.903882.
\]
Propositions~\ref{proposition: Lovasz number of G} and~\ref{proposition: Lovasz number of H} 
imply that this bound is strict, so the ratio bound is not tight for either graph. 
This is consistent with the fact that neither $G$ nor $H$ is edge-transitive, since edge-transitivity 
would imply tightness of the ratio bound (see Theorem~9 of \cite{Lovasz1979}).

The seed graphs can also be distinguished by their fractional
chromatic numbers. Recall that a fractional coloring assigns
nonnegative weights to independent sets so that the total weight
of the sets containing each vertex is at least one. The fractional
chromatic number \(\chi_f\) is the minimum total weight of such a
coloring. For the seed graphs,
\begin{equation}
\label{eq:fractional-seeds}
\chi_f(G)=\tfrac72,\qquad \chi_f(H)=\tfrac{10}{3}.
\end{equation}
Since \(G_n=G\vee K_{n-10}\) and \(H_n=H\vee K_{n-10}\), and the
fractional chromatic number is additive under graph joins, it follows
that
\begin{equation}
\label{eq:fractional-extensions}
\begin{aligned}
\chi_f(G_n)=\chi_f(G)+(n-10)=n-\tfrac{13}{2}, \qquad 
\chi_f(H_n)=\chi_f(H)+(n-10)=n-\tfrac{20}{3}.
\end{aligned}
\end{equation}
Thus, their fractional chromatic numbers differ by \(\tfrac16\),
although their ordinary chromatic numbers, and those of their
complements, coincide.

For completeness, the values in \eqref{eq:fractional-seeds} admit
short exact certificates. 
In $G$, the set $S=\{a,b,c,d,e,g,h\}$ induces a clique
expansion of $C_5$. Indeed, $a$ and $d$ are adjacent true twins
in $G[S]$, as are $c$ and $e$, since
\( N_{G[S]}(a)\setminus\{d\} = N_{G[S]}(d)\setminus\{a\} = \{b,h\} \),
and \( N_{G[S]}(c)\setminus\{e\} = N_{G[S]}(e)\setminus\{c\} = \{b,g\} \).
Collapsing each pair into a single vertex yields the cycle
\(
\{a,d\},\ b,\ \{c,e\},\ g,\ h,\ \{a,d\}.
\)
Thus, $G[S]$ is obtained from $C_5$ by replacing two
nonadjacent vertices with cliques of order two.
Since this operation preserves the independence number,
$\alpha(G[S])=\alpha(C_5)=2$. Therefore,
\[
\chi_f(G)\geq\chi_f(G[S])\geq\tfrac72.
\]
A matching fractional coloring assigns weight \(\tfrac12\) to each of
\[
\{a,e,i\},\quad \{a,g,j\},\quad \{b,f,h\},\quad
\{b,g,j\},\quad \{e,f,h\},
\]
and weight \(1\) to \(\{c,d,i\}\), with all other weights zero.
For \(H\),
\[
\chi_f(H)\geq\frac{|V(H)|}{\alpha(H)}=\frac{10}{3},
\]
and the bound is attained by assigning weight \(\tfrac13\) to each of
\[
\{a,c,i\},\quad \{a,e,f\},\quad \{a,f,j\},\quad
\{b,f,g\},\quad \{c,d,g\},\quad \{c,g,h\},
\]
and weight \(\tfrac23\) to each of \(\{b,h,j\}\) and \(\{d,e,i\}\).
Every listed set is independent in the corresponding graph, and the
weights cover every vertex with total weight at least one.
Furthermore, an independent set in a join lies entirely in one of
its parts, so the fractional chromatic number is additive under joins.
Applying \(\chi_f(F\vee K_t)=\chi_f(F)+t\) with \(t=n-10\) proves
\eqref{eq:fractional-extensions}.

Although fractional chromatic numbers distinguish these graph
families, the separation by the Lov\'asz number remains
computationally significant. For general input graphs, computing
the fractional chromatic number exactly is
NP-hard~\cite{GrotschelLovaszSchrijver1981,BacikMahajan1995}, despite
its linear programming formulation, which has a variable for each
independent set. In contrast, the Lov\'asz number admits a
semidefinite programming formulation of polynomial size and can be
approximated to additive error \(\varepsilon\), for
\(0<\varepsilon<1\), in time polynomial in the graph order and
\(\log(1/\varepsilon)\)~\cite{GrotschelLovaszSchrijver1981}.
This concerns numerical approximation rather than exact symbolic
computation. The uniform positive gap in \eqref{eq:separation-final}
therefore allows the two families to be distinguished by
polynomial-time approximation of their Lov\'asz numbers.
The NP-hardness statement concerns arbitrary input graphs; the
explicit certificates above make the fractional chromatic numbers
of this particular construction easy to evaluate.

Finally, we discuss the minimum order of the seed graphs required
by this construction. The proof of Theorem~\ref{theorem: all-orders}
starts with a pair of connected, regular, cospectral, and
nonisomorphic graphs on \(10\) vertices. It produces a pair of
connected, irregular, and nonisomorphic graphs on \(n\geq11\)
vertices that are simultaneously cospectral with respect to the
adjacency, Laplacian, signless Laplacian, and normalized Laplacian
matrices. This construction cannot start with a pair of regular
graphs on fewer than \(10\) vertices, since every regular graph
of order less than \(10\) is determined by its
spectrum~\cite{vanDamHaemers2003}.

\section{Summary and Conclusions}
\label{section: summary and conclusions}

This work establishes that connected, irregular, nonisomorphic graphs
can have identical adjacency, Laplacian, signless Laplacian, normalized
Laplacian, and Seidel spectra, as well as equal independence, clique,
chromatic, complement chromatic, and maximum-cut numbers, while having
distinct Lov\'{a}sz numbers. Such pairs exist on \(n\) vertices for every
integer \(n\geq 10\), with ten being the smallest possible order.
The existence result combines the construction for every \(n \geq 11\)
in Theorem~\ref{theorem: all-orders} with a separate irregular pair
of order ten, whereas minimality follows from an exhaustive computer-assisted
comparison of spectral and combinatorial invariants over connected,
irregular graphs of order at most nine (see 
Theorems~\ref{theorem: graphs of order at most 9} and~\ref{theorem: minimum-order}).
Seidel cospectrality need not be imposed in the nonexistence argument,
since the spectral and combinatorial conditions in
Theorem~\ref{theorem: graphs of order at most 9} already exclude
such pairs on at most nine vertices.

Theorem~\ref{theorem: all-orders} extends and strengthens an existence
result obtained in Theorem~4.19 of \cite{Sason2024}, where the neighbors 
splitting join operation yielded pairs for even orders \(n\geq 14\). 
The present construction uses ordinary joins to produce pairs for
every \(n\geq 11\), removing the parity restriction and additionally
establishing equality of the Seidel spectra, complement chromatic
numbers, and maximum-cut numbers, together with exact analytical
evaluations of the Lov\'{a}sz numbers. 

The countably infinite family constructed in the proof of Theorem~\ref{theorem: all-orders} 
is obtained by joining each of the two regular seed graphs of order ten due 
to van Dam and Haemers~\cite{vanDamHaemers2003} with \(K_{n-10}\), for 
each integer \(n\geq 11\). The seed graphs themselves satisfy all 
the stated properties except irregularity. For these seed graphs, the 
join construction preserves equality of the five spectra and the stated 
integer-valued invariants within each pair, while leaving the respective 
Lov\'{a}sz numbers unchanged. Cardinality-constrained maximum-cut profiles 
yield exact formulas and establish equality of the maximum-cut numbers 
within each pair, both for these joins and for the corresponding joins 
formed from the complements of the seed graphs.

The exact evaluations in Propositions~\ref{proposition: Lovasz number of G}
and~\ref{proposition: Lovasz number of H} provide more than the strict
separation needed for the existence result in Theorem~\ref{theorem: all-orders}.
In particular, the exact determination of \(\vartheta(H)\) is supported
by matching primal and dual certificates.
Graph automorphisms and complementary slackness guide the algebraic
construction of the dual certificate, while symmetry reduces the verification 
of positive semidefiniteness to smaller invariant blocks.
The algebraic details are provided in Appendices~\ref{app:quartic-equation}
and~\ref{app:H-block-diagonalization}.
These calculations illustrate how symmetry, complementary slackness,
and algebraic identities can yield exact certificates of optimality
for semidefinite programs.

For the irregular pair of order ten, Appendix~\ref{app:irregular-theta-separation} 
establishes strict separation through an analytical upper bound and an explicit 
rational semidefinite certificate, independently of numerical optimization
tolerances.

The uniform positive gap in \eqref{eq:separation-final} allows the members 
of each pair in the countably infinite family to be distinguished by 
approximating their Lov\'{a}sz numbers to a fixed additive accuracy.
Polynomial-time approximability therefore provides an efficient
means of certifying nonisomorphism throughout this family.
Thus, the Lov\'{a}sz number detects distinctions not determined
by the combined spectral and integer-valued invariants considered
here. However, equality of Lov\'{a}sz numbers does not, in general,
imply graph isomorphism.

\appendix

\setcounter{section}{0}
\section{Derivation of the Unique Solution of Equation 
\texorpdfstring{\eqref{eq:H-rho-polynomial}}{(\ref*{eq:H-rho-polynomial})} 
in \texorpdfstring{$(0,\tfrac12)$}{(0, 1/2)}}
\label{app:quartic-equation}

\setcounter{equation}{0}
\renewcommand{\theequation}{\Alph{section}\arabic{equation}}
\makeatletter
\@addtoreset{equation}{section}
\makeatother

We derive the unique solution \(\rho^\ast\in(0,\tfrac12)\) of equation 
\eqref{eq:H-rho-polynomial}. Introduce the variable $t$ by 
\begin{equation}
\label{eq:t-rho-substitution}
\rho^\ast=\tfrac12 (t-1).
\end{equation}
Then, the condition \(\rho^\ast\in(0,\tfrac12)\) is equivalent to \(t\in(1,2)\), 
and substitution into \eqref{eq:H-rho-polynomial} gives the equation 
\begin{equation}
\label{eq:depressed-quartic}
t^4-14t^2+16t+5=0.
\end{equation}
We seek a factorization of the form
\begin{equation}
\label{eq:quartic-factorization-general}
t^4-14t^2+16t+5
= (t^2+2at+b)(t^2-2at+c).
\end{equation}
Comparison of the coefficients on both sides of
\eqref{eq:quartic-factorization-general} yields
\begin{equation}
\label{eq:factorization-system}
b+c-4a^2=-14,
\qquad
2a(c-b)=16,
\qquad
bc=5.
\end{equation}
Let \( a=\sqrt{z} \).
It follows from the first two equations in
\eqref{eq:factorization-system} that
\[
b+c=4z-14,
\qquad
c-b=\frac{8}{\sqrt{z}},
\]
and hence
\begin{equation}
\label{eq:b-c-values}
b=2z-7-\frac{4}{\sqrt{z}}, \qquad c=2z-7+\frac{4}{\sqrt{z}}.
\end{equation}
The remaining condition \(bc=5\) becomes
\[
(2z-7)^2-\frac{16}{z}=5,
\]
which, after rearranging terms, gives the resolvent cubic
\begin{equation}
\label{eq:resolvent-cubic}
z^3-7z^2+11z-4=0.
\end{equation}
To solve \eqref{eq:resolvent-cubic}, set
\begin{equation}
\label{eq:z-u-substitution}
z=u+\tfrac{7}{3}.
\end{equation}
The resulting depressed cubic is
\begin{equation}
\label{eq:depressed-cubic}
27u^3-144 u-101=0.
\end{equation}
Substituting \( u=\tfrac{8}{3} \cos\theta \)
into \eqref{eq:depressed-cubic}, and using the identity
\( \cos(3\theta)=4\cos^3\theta-3\cos\theta \), gives
\begin{equation}
\cos(3\theta)=\frac{101}{128}. 
\end{equation}
By \eqref{eq:z-u-substitution} and the last equation, the required solution $z$ 
of the resolvent cubic in \eqref{eq:resolvent-cubic} is given in \eqref{eq: z}. 
Using \eqref{eq:b-c-values}, the quartic in \eqref{eq:depressed-quartic} factors as
\begin{align}
t^4-14t^2+16t+5
= \left( t^2+2\sqrt{z}\,t+2z-7-\frac{4}{\sqrt{z}} \right)
\cdot \left( t^2-2\sqrt{z}\,t+2z-7+\frac{4}{\sqrt{z}} \right).
\label{eq:quartic-factorization-final}
\end{align}
The root $t \in (1,2)$ of the second quadratic factor is given by 
\begin{equation}
\label{eq:t-star}
t = \sqrt{z} - \sqrt{7-z-4z^{-1/2}},
\end{equation}
with \(z\) in \eqref{eq: z}.
Combining \eqref{eq:t-rho-substitution} and \eqref{eq:t-star} 
gives \eqref{eq: rho}.

\section{Details of the Block Diagonalization of
\texorpdfstring{\(\boldsymbol{M}\)}{M}}
\label{app:H-block-diagonalization}

In this appendix, we provide the details of the symmetry reduction used
in Step~4 and derive the four matrices in
\eqref{eq:H-block-pp}--\eqref{eq:H-block-mm}.
For \(\nu\in[6]\), define
\begin{align}
\boldsymbol{T}_\nu \coloneqq
\sum_{\{x,y\} \in \mathcal{E}_\nu}
\bigl( \boldsymbol{E}_{x,y}+\boldsymbol{E}_{y,x} \bigr),
\end{align}
where the six edge classes of $H$, $\mathcal{E}_\nu$ for $\nu \in [6]$, 
are given in \eqref{eq:H-edge-classes}.
Then, by \eqref{eq:H-dual-matrix-general},
\begin{align}
\label{eq: M}
\boldsymbol{M} = \tau\boldsymbol{I}_{10}
-\boldsymbol{J}_{10} + \sum_{\nu=1}^{6} w_\nu \boldsymbol{T}_\nu.
\end{align}

\subsection{The Common Eigenspaces}

We first justify the orthogonal decomposition used in Step~4. Let
\(\boldsymbol{P}_\pi\) and \(\boldsymbol{P}_\sigma\) denote the
permutation matrices associated with \(\pi\) and \(\sigma\),
respectively. Since \(\pi\) and \(\sigma\) are commuting involutions,
\begin{align}
\boldsymbol{P}_\pi^2 = \boldsymbol{P}_\sigma^2 = \boldsymbol{I}_{10},
\qquad
\boldsymbol{P}_\pi \boldsymbol{P}_\sigma = \boldsymbol{P}_\sigma \boldsymbol{P}_\pi.
\end{align}
Moreover, both permutation matrices are symmetric and orthogonal.
They are therefore simultaneously orthogonally diagonalizable, and
their common eigenspaces are given in \eqref{eq: W-subspaces}.
Consequently, \eqref{eq: direct sum} holds, and its summands
are pairwise orthogonal with respect to the standard Euclidean
inner product on $\mathbb{R}^{10}$.

Since each edge class \(\mathcal{E}_\nu\) is invariant under both
\(\pi\) and \(\sigma\), one has
\begin{align}
\label{eq2: 18.09.26}
\boldsymbol{P}_\pi\boldsymbol{T}_\nu
= \boldsymbol{T}_\nu\boldsymbol{P}_\pi,
\qquad
\boldsymbol{P}_\sigma\boldsymbol{T}_\nu
= \boldsymbol{T}_\nu\boldsymbol{P}_\sigma,
\qquad \nu\in[6].
\end{align}
The matrices \(\boldsymbol{I}_{10}\) and \(\boldsymbol{J}_{10}\) also
commute with every permutation matrix, so it follows from
\eqref{eq: M} and \eqref{eq2: 18.09.26} that
\begin{align}
\label{eq3: 18.09.26}
\boldsymbol{P}_\pi\boldsymbol{M}
= \boldsymbol{M}\boldsymbol{P}_\pi,
\qquad
\boldsymbol{P}_\sigma\boldsymbol{M}
= \boldsymbol{M}\boldsymbol{P}_\sigma.
\end{align}
Thus, if
\(\underline{v}\in W_{\varepsilon,\delta}\), then
by \eqref{eq: W-subspaces} and \eqref{eq3: 18.09.26}
\begin{align}
\label{eq4: 18.09.26}
\boldsymbol{P}_\pi(\boldsymbol{M}\underline{v})
= \boldsymbol{M}(\boldsymbol{P}_\pi\underline{v})
= \varepsilon\boldsymbol{M}\underline{v}
\end{align}
and, similarly,
\( \boldsymbol{P}_\sigma(\boldsymbol{M}\underline{v})
= \delta\boldsymbol{M}\underline{v} \).
Therefore,
\begin{align}
\label{eq5: 18.09.26}
\boldsymbol{M} W_{\varepsilon,\delta}
\subseteq W_{\varepsilon,\delta},
\qquad
\varepsilon,\delta\in\{+1,-1\}.
\end{align}

We next verify the bases used for these common eigenspaces. From the
definitions of \(\pi\) and \(\sigma\), and by \eqref{eq: A,D}--\eqref{eq: B}, 
direct calculation gives
\begin{equation}
\label{eq6: 18.09.26}
\boldsymbol{P}_\pi \, \underline{u}_{\varepsilon,\delta}
=\varepsilon \underline{u}_{\varepsilon,\delta}, \qquad 
\boldsymbol{P}_\sigma \, \underline{u}_{\varepsilon,\delta}
=\delta \underline{u}_{\varepsilon,\delta}, \qquad 
\boldsymbol{P}_\pi \, \underline{v}_{\varepsilon,\delta}
=\varepsilon \underline{v}_{\varepsilon,\delta}, \qquad 
\boldsymbol{P}_\sigma \, \underline{v}_{\varepsilon,\delta}
=\delta \underline{v}_{\varepsilon,\delta},
\end{equation}
and
\begin{equation}
\label{eq7: 18.09.26}
\boldsymbol{P}_\pi \, \underline{t}_{+,\delta} = \underline{t}_{+,\delta},
\qquad
\boldsymbol{P}_\sigma \, \underline{t}_{+,\delta} = \delta \underline{t}_{+,\delta}.
\end{equation}
Hence,
\begin{equation}
\label{eq8: 18.09.26}
\underline{u}_{\varepsilon,\delta}, \underline{v}_{\varepsilon,\delta}
\in W_{\varepsilon,\delta},
\qquad \underline{t}_{+,\delta} \in W_{+,\delta}.
\end{equation}

\noindent 
The vectors
\[
\bigl\{ \underline{u}_{\varepsilon,\delta}, \, \underline{v}_{\varepsilon,\delta}:
\varepsilon, \delta \in \{+1,-1\} \bigr\}
\cup \bigl\{ \underline{t}_{+,+}, \, \underline{t}_{+,-} \bigr\}
\]
form an orthonormal set of ten vectors in \(\mathbb{R}^{10}\).
They therefore form an orthonormal basis of \(\mathbb{R}^{10}\).
It follows that the ordered tuples
\begin{align}
\label{eq9: 18.09.26}
(\underline{u}_{++}, \, \underline{t}_{++}, \, \underline{v}_{++}), \qquad
(\underline{u}_{+-}, \, \underline{t}_{+-}, \, \underline{v}_{+-}), \qquad
(\underline{u}_{-+}, \, \underline{v}_{-+}), \qquad
(\underline{u}_{--}, \, \underline{v}_{--})
\end{align}
are orthonormal bases of
\(W_{++}, \, W_{+-}, \, W_{-+}\), and \(W_{--}\), respectively.

\subsection{Computation of the Four Blocks}

Due to the orthonormality of the bases above, the entries of the matrix
representing the restriction of \(\boldsymbol{M}\) to any one of the
four invariant subspaces are the corresponding inner products
\( \langle \underline{u},\boldsymbol{M}\underline{v}\rangle \).
For \(\varepsilon,\delta\in\{+1,-1\}\), set
\begin{align}
\label{eq10: 18.09.26}
s_{\varepsilon,\delta} \coloneqq
\tfrac12 \, (1+\varepsilon)(1+\delta).
\end{align}
The sum of the coordinates of each of
\(\underline{u}_{\varepsilon,\delta}\) and
\(\underline{v}_{\varepsilon,\delta}\) is \(s_{\varepsilon,\delta}\).
Thus, \( \langle \underline{1}, 
\underline{u}_{\varepsilon,\delta} \rangle = s_{\varepsilon,\delta} 
= \langle \underline{1}, \underline{v}_{\varepsilon,\delta} \rangle \), and 
since \( \langle \underline{u}, \boldsymbol{J}_{10} \, \underline{v} \rangle = 
\langle \underline{u}, \underline{1} \rangle \, \langle \underline{1}, \underline{v} \rangle \)
for all \( \underline{u}, \underline{v} \in \mathbb{R}^{10} \), we get 
\begin{align}
\label{eq11: 18.09.26}
\langle \underline{u}_{\varepsilon,\delta},
\boldsymbol{J}_{10} \, \underline{u}_{\varepsilon,\delta}\rangle
= \langle \underline{u}_{\varepsilon,\delta},
\boldsymbol{J}_{10} \, \underline{v}_{\varepsilon,\delta}\rangle
= \langle \underline{v}_{\varepsilon,\delta},
\boldsymbol{J}_{10} \, \underline{v}_{\varepsilon,\delta}\rangle
=s_{\varepsilon,\delta}^{\,2}.
\end{align}
Since \(\underline{u}_{\varepsilon,\delta}\) and
\(\underline{v}_{\varepsilon,\delta}\) are orthonormal vectors,
\begin{align}
\label{eq12: 18.09.26}
\langle \underline{u}_{\varepsilon,\delta},
\boldsymbol{I}_{10} \, \underline{u}_{\varepsilon,\delta}\rangle = 1, \qquad 
\langle \underline{u}_{\varepsilon,\delta},
\boldsymbol{I}_{10} \, \underline{v}_{\varepsilon,\delta}\rangle = 0, \qquad 
\langle \underline{v}_{\varepsilon,\delta},
\boldsymbol{I}_{10} \, \underline{v}_{\varepsilon,\delta}\rangle = 1.
\end{align}

We now calculate the contributions from the edge classes in
\eqref{eq:H-edge-classes}. It can be verified that the only 
nonzero contributions to the three inner products above are
given by 
\begin{align}
\label{eq13: 18.09.26}
\begin{aligned}
& \langle \underline{u}_{\varepsilon,\delta},
\boldsymbol{T}_3 \, \underline{u}_{\varepsilon,\delta}\rangle
=\varepsilon\delta, \qquad 
&\langle \underline{u}_{\varepsilon,\delta},
\boldsymbol{T}_2 \, \underline{v}_{\varepsilon,\delta}\rangle 
= 1, \\ 
& \langle \underline{u}_{\varepsilon,\delta},
\boldsymbol{T}_4 \, \underline{v}_{\varepsilon,\delta}\rangle
=\delta, \qquad 
&\langle \underline{v}_{\varepsilon,\delta},
\boldsymbol{T}_6 \, \underline{v}_{\varepsilon,\delta}\rangle
=\delta.
\end{aligned}
\end{align}
Consequently,
\begin{align}
\langle \underline{u}_{\varepsilon,\delta},
\boldsymbol{M}\underline{u}_{\varepsilon,\delta}\rangle
&=
\tau-s_{\varepsilon,\delta}^{\,2}
+\varepsilon\delta w_3,
\label{eq:H-appendix-entry-AA}\\
\langle \underline{u}_{\varepsilon,\delta},
\boldsymbol{M}\underline{v}_{\varepsilon,\delta}\rangle
&=
w_2+\delta w_4-s_{\varepsilon,\delta}^{\,2},
\label{eq:H-appendix-entry-AD}\\
\langle \underline{v}_{\varepsilon,\delta},
\boldsymbol{M}\underline{v}_{\varepsilon,\delta}\rangle
&=
\tau-s_{\varepsilon,\delta}^{\,2}
+\delta w_6.
\label{eq:H-appendix-entry-DD}
\end{align}

It remains to calculate the entries involving \(\underline{t}_{+,\delta}\), which
appears only in the bases of \(W_{++}\) and \(W_{+-}\). The sum of the
coordinates of \(\underline{t}_{+,\delta}\) is equal to \( \frac{\sqrt{2}}{2} (1+\delta). \)
Since \(s_{+,\delta}=1+\delta\), it follows that
\begin{align}
& \langle \underline{t}_{+,\delta},
\boldsymbol{J}_{10} \, \underline{t}_{+,\delta}\rangle
=1+\delta, \\
& \langle \underline{u}_{+,\delta},
\boldsymbol{J}_{10} \, \underline{t}_{+,\delta}\rangle
=\sqrt{2} \, (1+\delta), \\
& \langle \underline{t}_{+,\delta},
\boldsymbol{J}_{10} \, \underline{v}_{+,\delta}\rangle
=\sqrt{2} \, (1+\delta),
\end{align}
where we rely on the equality  \( \tfrac12 \, (1+\delta)^2 = 1+\delta \) 
for \(\delta \in \{+1,-1\}\).
It can be verified that the only relevant contributions from the edge classes are
\begin{equation}
\langle \underline{u}_{+,\delta},
\boldsymbol{T}_1 \, \underline{t}_{+,\delta}\rangle
= \sqrt{2}, \qquad 
\langle \underline{t}_{+,\delta},
\boldsymbol{T}_5 \, \underline{v}_{+,\delta}\rangle
= \sqrt{2}.
\end{equation}
Therefore,
\begin{align}
\langle \underline{t}_{+,\delta},
\boldsymbol{M}\underline{t}_{+,\delta}\rangle
&=
\tau-(1+\delta),
\label{eq:H-appendix-entry-BB}\\
\langle \underline{u}_{+,\delta},
\boldsymbol{M}\underline{t}_{+,\delta}\rangle
&=
\sqrt{2}\bigl(w_1-(1+\delta)\bigr),
\label{eq:H-appendix-entry-AB}\\
\langle \underline{t}_{+,\delta},
\boldsymbol{M}\underline{v}_{+,\delta}\rangle
&=
\sqrt{2}\bigl(w_5-(1+\delta)\bigr).
\label{eq:H-appendix-entry-BD}
\end{align}

We can now recover the four matrices displayed in Step~4. If
\(\varepsilon=+1\), the restriction of \(\boldsymbol{M}\) to
\(W_{+,\delta}\), with respect to the ordered basis
\( (\underline{u}_{+,\delta},\underline{t}_{+,\delta},\underline{v}_{+,\delta}) \),
is represented by
\begin{equation}
\label{eq1: 3.10.26}
\begin{pmatrix}
\tau-(1+\delta)^2+\delta w_3
& \sqrt{2}\bigl(w_1-(1+\delta)\bigr)
& w_2+\delta w_4-(1+\delta)^2 \\
\sqrt{2}\bigl(w_1-(1+\delta)\bigr)
& \tau-(1+\delta)
& \sqrt{2}\bigl(w_5-(1+\delta)\bigr) \\
w_2+\delta w_4-(1+\delta)^2
& \sqrt{2}\bigl(w_5-(1+\delta)\bigr)
& \tau-(1+\delta)^2+\delta w_6
\end{pmatrix}.
\end{equation}
Substituting \(\delta=+1\) into \eqref{eq1: 3.10.26} gives
\(\boldsymbol{M}_{++}\) in \eqref{eq:H-block-pp}, and substituting 
\(\delta=-1\) into \eqref{eq1: 3.10.26} gives
\(\boldsymbol{M}_{+-}\) in \eqref{eq:H-block-pm}.

If \(\varepsilon=-1\), then
\(s_{-1,\delta}=0\), and the restriction of \(\boldsymbol{M}\) to
\(W_{-,\delta}\), with respect to the ordered basis
\(
(\underline{u}_{-,\delta},\underline{v}_{-,\delta}),
\)
is represented by
\begin{equation}
\label{eq2: 3.10.26}
\begin{pmatrix}
\tau-\delta w_3 & w_2+\delta w_4\\
w_2+\delta w_4 & \tau+\delta w_6
\end{pmatrix}.
\end{equation}
Substituting \(\delta=+1\) into \eqref{eq2: 3.10.26} gives
\(\boldsymbol{M}_{-+}\) in \eqref{eq:H-block-mp}, and substituting
\(\delta=-1\) into \eqref{eq2: 3.10.26} gives
\(\boldsymbol{M}_{--}\) in \eqref{eq:H-block-mm}.

Finally, let \(\boldsymbol{Q}\) be the orthogonal matrix whose columns
are the basis vectors in the order
\[
\underline{u}_{++}, \, \underline{t}_{++}, \, \underline{v}_{++}, \, 
\underline{u}_{+-}, \, \underline{t}_{+-}, \, \underline{v}_{+-}, \, 
\underline{u}_{-+}, \, \underline{v}_{-+}, \, 
\underline{u}_{--}, \, \underline{v}_{--}.
\]
The invariance of the four common eigenspaces gives
\[
\boldsymbol{Q}^{\mathsf T}\boldsymbol{M}\boldsymbol{Q}
=
\boldsymbol{M}_{++}
\oplus
\boldsymbol{M}_{+-}
\oplus
\boldsymbol{M}_{-+}
\oplus
\boldsymbol{M}_{--}.
\]
Since \(\boldsymbol{Q}\) is orthogonal,
\[
\boldsymbol{M}\succeq\boldsymbol{0}
\quad\Longleftrightarrow\quad
\boldsymbol{M}_{++}\succeq\boldsymbol{0},\quad
\boldsymbol{M}_{+-}\succeq\boldsymbol{0},\quad
\boldsymbol{M}_{-+}\succeq\boldsymbol{0},\quad
\boldsymbol{M}_{--}\succeq\boldsymbol{0}.
\]
This verifies the block diagonalization and all the entries in
\eqref{eq:H-block-pp}--\eqref{eq:H-block-mm}.

\section{Analytical Separation of the Lov\'{a}sz Numbers of the Irregular Graphs of Order Ten in Figure~\ref{fig:G_and_H_irregular_graphs}}
\label{app:irregular-theta-separation}

In this appendix, we prove analytically that the graphs \(G^\ast\) and
\(H^\ast\) in Figure~\ref{fig:G_and_H_irregular_graphs} have distinct
Lov\'{a}sz numbers. This rigorously justifies the
separation used in the proof of Theorem~\ref{theorem: minimum-order},
independently of the numerical values in \eqref{eq1: 19.10.26}. 
Throughout this appendix, the vertex labels are those of
Figure~\ref{fig:G_and_H_irregular_graphs}; in particular,
\(V(G^\ast)=V(H^\ast)=\{0,1,\ldots,9\}\).

\begin{proposition}
\label{prop:irregular-theta-separation}
{\em The graphs \(G^\ast\) and \(H^\ast\) in
Figure~\ref{fig:G_and_H_irregular_graphs} satisfy
\begin{equation}
\label{eq:irregular-theta-separation}
\vartheta(G^\ast) \leq 2+\sqrt{5}
<\frac{14345}{3337} \leq \vartheta(H^\ast).
\end{equation}}
\end{proposition}

\begin{proof}
We first apply the partition inequality \eqref{eq:partition-bound}
and the adjacent-true-twin reduction \eqref{eq3: 1.10.26} to obtain an
upper bound on \(\vartheta(G^\ast)\). We then construct an explicit rational
feasible matrix for \(H^\ast\) in the semidefinite characterization
\eqref{eq:theta-sdp}.

\medskip
\noindent\textbf{Step 1: An upper bound on \(\vartheta(G^\ast)\).}
Consider the partition
\begin{equation}
\label{eq:irregular-theta-partition}
S_1=\{0,4\},\qquad S_2=\{3,7\},\qquad
S_3=\{1,2,5,6,8,9\}.
\end{equation}
The induced subgraphs \(G^\ast[S_1]\) and \(G^\ast[S_2]\) are both isomorphic
to \(K_2\), and hence have Lov\'{a}sz number \(1\).
In the induced subgraph \(G^\ast[S_3]\), vertices \(2\) and \(9\) are
adjacent true twins: their closed neighborhoods within this subgraph
are both \(\{2,6,8,9\}\). Deleting vertex \(9\) leaves precisely
the five-cycle with cyclic vertex order \( 1,\ 5,\ 8,\ 2,\ 6 \).
Consequently, Lemma~\ref{lemma: adjacent true twins}, specifically
\eqref{eq3: 1.10.26}, gives
\begin{equation}
\label{eq:irregular-theta-cycle}
\vartheta(G^\ast[S_3])
=\vartheta\bigl(G^\ast[S_3\setminus\{9\}]\bigr)
=\vartheta(C_5)=\sqrt{5}.
\end{equation}
Applying \eqref{eq:partition-bound} twice to the partition
\eqref{eq:irregular-theta-partition}, we obtain
\begin{equation}
\label{eq:irregular-theta-upper}
\vartheta(G^\ast)
\leq\vartheta(G^\ast[S_1])+\vartheta(G^\ast[S_2])+\vartheta(G^\ast[S_3])
=2+\sqrt{5}.
\end{equation}

\medskip
\noindent\textbf{Step 2: A rational lower bound on \(\vartheta(H^\ast)\).}
With rows and columns indexed by \(0,1,\ldots,9\), in this order,
define the symmetric integer matrix
\begin{equation}
\label{eq:irregular-theta-certificate}
\begingroup
\setlength{\arraycolsep}{3pt}
\renewcommand{\arraystretch}{1.12}
\boldsymbol{R}=\begin{pmatrix}
612 & 453 & 131 & 96 & 0 & 0 & 295 &485 & 559 & 0\\
453 & 1103 & 812 & 328 & 649 & 0 & 0 & 474 & 364 & 559\\
131 & 812 & 1076 & 850 & 944 & 449 & 0 & 0 & 0 & 364\\
96 & 328 & 850 & 1103 & 1006 & 910 & 449 & 0 & 0 & 0\\
0 & 649 & 944 & 1006 & 1715 & 1102 & 780 & 617 & 0 & 559\\
0 & 0 & 449 & 910 & 1102 & 1103 & 850 & 328 & 0 & 0\\
295 & 0 & 0 & 449 & 780 & 850 & 1076 & 812 & 364 & 0\\
485 & 474 & 0 & 0 & 617 & 328 & 812 & 1103 & 559 & 364\\
559 & 364 & 0 & 0 & 0 & 0 & 364 & 559 & 560 & 0\\
0 & 559 & 364 & 0 & 559 & 0 & 0 & 364 & 0 & 560
\end{pmatrix}.
\endgroup
\end{equation}
We verify all three feasibility conditions in \eqref{eq:theta-sdp}
for a normalized version of \(\boldsymbol{R}\).

\begin{enumerate}
\item The edge set of \(H^\ast\) is
\begin{align*}
E(H^\ast)=\bigl\{&\{0,4\},\{0,5\},\{0,9\},\{1,5\},\{1,6\},
             \{2,6\},\{2,7\},\{2,8\},\\
            &\{3,7\},\{3,8\},\{3,9\},\{4,8\},\{5,8\},
             \{5,9\},\{6,9\},\{8,9\}\bigr\}.
\end{align*}
Inspection of \eqref{eq:irregular-theta-certificate} shows that
\(R_{ij}=0\) for every \(\{i,j\}\in E(H^\ast)\).

\item The matrix \(\boldsymbol{R}\) is positive definite. To verify this
with exact arithmetic, let
\[
\Delta_k=\det\bigl((R_{ij})_{0\leq i,j\leq k-1}\bigr),
\qquad k\in\{1,\ldots,10\},
\]
be its leading principal minors. Their exact values are
\begin{center}
\begin{tabular}{r r@{\qquad}r r}
\toprule
\(k\)&\(\Delta_k\)&\(k\)&\(\Delta_k\)\\
\midrule
1 & 612 & 6 & 5130721204529\\
2 & 469827 & 7 & 12997260201813\\
3 & 179459773 & 8 & 26766721732696\\
4 & 16187846247 & 9 & 159752802277260\\
5 & 1464903945776 & 10 & 224469839133576\\
\bottomrule
\end{tabular}
\end{center}
Since all these minors are strictly positive, Sylvester's criterion
gives \(\boldsymbol{R}\succ0\).

\item Direct summation gives
\begin{equation}
\label{eq:irregular-theta-certificate-sums}
\tr(\boldsymbol{R})=10011,
\qquad
\langle\boldsymbol{J}_{10},\boldsymbol{R}\rangle=43035.
\end{equation}
It follows that 
\( 
\boldsymbol{X} \coloneqq \frac{1}{10011}\,\boldsymbol{R}
\)
is positive semidefinite, has trace one, and vanishes at every edge
of \(H^\ast\). 
\end{enumerate}
The matrix $\boldsymbol{X}$ is therefore feasible in \eqref{eq:theta-sdp}, 
yielding
\begin{equation}
\label{eq:irregular-theta-lower}
\vartheta(H^\ast) \geq
\langle\boldsymbol{J}_{10},\boldsymbol{X}\rangle
=\frac{43035}{10011}=\frac{14345}{3337}.
\end{equation}

\medskip
\noindent\textbf{Step 3: Strict separation.}
The two bounds are strictly separated, since 
\begin{equation}
\label{eq1: 20.10.26}
\frac{14345}{3337} = 4.298771351 \ldots > 
4.236067977 \ldots = 2+\sqrt5.
\end{equation}
This strict inequality can be verified using only integer arithmetic:
\[
\frac{14345}{3337} - 2 = \frac{7671}{3337} > 0,
\qquad
7671^2 - 5 \cdot 3337^2 = 3166396>0.
\]
Combining \eqref{eq:irregular-theta-upper}, \eqref{eq:irregular-theta-lower}, 
and \eqref{eq1: 20.10.26} proves \eqref{eq:irregular-theta-separation}.
\end{proof}

\begin{remark}
\label{rem:irregular-theta-separation}
{\em The upper bound on \(\vartheta(G^\ast)\) in Proposition~\ref{prop:irregular-theta-separation} 
agrees with its reported numerical 
value to the displayed precision in \eqref{eq1: 19.10.26}; 
no assertion of equality is needed for the separation. The lower 
bound on \(\vartheta(H^\ast)\) in Proposition~\ref{prop:irregular-theta-separation} 
is within \(0.003\) of its reported numerical value in \eqref{eq1: 19.10.26}.
The matrix in \eqref{eq:irregular-theta-certificate} was obtained by
rounding a numerically computed feasible solution and increasing its diagonal.
Its role in the proof of Proposition~\ref{prop:irregular-theta-separation}, 
however, is entirely exact: feasibility follows from the displayed integer 
entries and positive principal minors, without relying on numerical 
optimization tolerances.}
\end{remark}

\section*{Acknowledgments}
The author thanks the anonymous reviewers and the academic editor for their timely 
and constructive comments, which helped improve the presentation of this work.
The author used ChatGPT~5.6 (OpenAI) to assist with developing and checking
the SageMath code \cite{SageMath} used in this work, to provide an additional
check of selected analytical calculations, and to assist with language
editing and stylistic refinement of the manuscript. The author takes full
responsibility for the content of the manuscript.

\end{document}